\documentclass[12pt]{article}

\usepackage{layout}
\usepackage[normalem]{ulem}
\usepackage[dvipsnames]{xcolor}
\usepackage[OT1]{fontenc}
\usepackage{color}
\usepackage{amsthm,amsmath,graphicx,latexsym,amssymb,amscd,amsfonts,enumerate, bm}
\usepackage[colorlinks,citecolor=blue,linkcolor=red,urlcolor=blue]{hyperref}
\usepackage[labelfont=bf,labelsep=colon]{caption}
\usepackage[export]{adjustbox}

\theoremstyle{plain}
\newtheorem{theorem}{Theorem}
\newtheorem{proposition}[theorem]{Proposition}
\newtheorem{lemma}[theorem]{Lemma}
\newtheorem{corollary}[theorem]{Corollary}
\newtheorem{prop}[theorem]{Proposition}
\newtheorem{definition}[theorem]{Definition}
\newtheorem{ass}[theorem]{Assumption}
\newtheorem{remark}[theorem]{Remark}
\newtheorem{example}[theorem]{Example}

\def\R{\mathbb{R}}
\def\P{\mathbb{P}}

\def\ud{\mathrm{d}}

\def\tb{\textbf}

\def\supp{\mathrm{supp}\,}
\def\ind{{\mathchoice{1\mskip-4mu\mathrm l}{1\mskip-4mu\mathrm l}
{1\mskip-4.5mu\mathrm l}{1\mskip-5mu\mathrm l}}}

\begin{document}

\title{Optimal Selling of Defaultable Assets using the Distribution Builder}
\author{
Sixian Jin
\thanks{Department of Mathematics, California State University San Marcos, {\tt sjin@csusm.edu.}}
~~~~Stephan Sturm
\thanks{Department of Mathematical Sciences, Worcester Polytechnic Institute, {\tt ssturm@wpi.edu.}}
}

\date{October 19, 2026}

\maketitle

\begin{abstract}
We consider the problem of when it is best to sell a risky asset in the \textit{distribution builder} framework under the consideration of potential ruin. This approach allows investors to express their preferences directly as a desired target distribution without first specifying a risk aversion or utility function. Mathematically, the problem is closely related to the Skorokhod embedding problem of attaining a given distribution by stopping a diffusion process. We work in a general framework of one-dimensional diffusion processes and extend existing results to include the possibility of ruin. We first provide a full characterization of the set of distributions that can be attained before ruin occurs. Then, we formulate two optimization problems that tackle the issue of what to do if the originally specified distribution is either not attainable or not optimal: finding the attainable distribution closest to an unattainable distribution and selecting an optimal attainable distribution under first-order stochastic dominance constraints if the originally specified distribution is attainable, but not optimal. We provide sufficient conditions for the existence and uniqueness of solutions to these constrained convex optimization problems in an extended \(f\)-divergence framework, provide an analytic characterization of the solutions, and present numerical examples.
\end{abstract}

\vspace{5mm}
 
\begin{flushleft}
	 %\textbf{JEL classification:}  G11, C61.
\end{flushleft}

\section{Introduction}\label{sec:intro}

We consider the problem of selling a defaultable asset prone to default in the \textit{distribution builder} framework. Determining the optimal timing for the sale of an indivisible asset is a mainstay problem in mathematical finance. For example, Shiryaev, Xu and Zhou \cite{SXZ08} and du Toit and Peskir \cite{dTP09} consider the problem for asset prices modeled by a geometric Brownian motion to find, on a finite time horizon, the stopping time that gives a payoff closest to the ultimate maximum -- independently of the seller's risk preferences. A preference-dependent approach has been developed by Leung and Wang \cite{LW19} for classical expected utility and geometric Brownian motion and exponential Ornstein--Uhlenbeck processes as well as by Pedersen and Peskir \cite{PP16} for a mean-variance criterion for geometric Brownian motion. General diffusion models have been used by Henderson \cite{Hen12} in a prospect-theory framework and by Carr and Sturm \cite{CS26} for a distribution builder approach.

The distribution builder framework was first established by Sharpe, Goldstein, and Blythe \cite{SGB00} in 2000 to provide an alternative to the expected utility theory (EUT) framework. They address the issue that the EUT is widely used in academia thanks to its crisp axiomatic approach and tractability, but industry has been sluggish to adopt it. At the core of this is that utility functions (or even risk aversions) are hard to establish consistently from empirical data. The authors propose, in an optimal investment setting, to bypass the step of the specification of the investor's utility function and suggest eliciting preferences directly by letting the investor determine directly their desired target distribution. The framework has been further discussed in \cite{S07, GJS08}; a similar idea, based on Markowitz's mean-variance setting and mental accounting has been proposed in \cite{DMSS10}. While this direction of research has certainly not been the mainstream approach, it has gained also interest from institutional portfolio management in the framework of goal-based allocation, see Sepp \cite{Sep26}. The distribution builder framework has been used for the optimal selling problem by Carr and Sturm \cite{CS26}, substituting the stopping time of the sale for the terminal time of the portfolio optimization as the time when the target distribution is reached.

From a mathematical perspective, this problem is closely related to the classical Skorokhod embedding problem. Given a stochastic process $R=(R_t)_{t\geq 0}$ and a probability measure $\mu$, the Skorokhod embedding problem asks whether there exists a stopping time $\tau$ such that
\[
 R_\tau \sim \mu.
\]
If such a stopping time exists and satisfies some finiteness or integrability conditions, then $\mu$ is said to be \textit{attainable} for the process $R$. The problem was originally formulated for Brownian motion by Skorokhod \cite{Sko65} and has since been studied for many classes of processes, including Brownian motion with drift \cite{GF00}, diffusions \cite{PP01, CH04}, and L\'evy processes \cite{BL92, OP09, DGPR19}; an excellent survey on the topic is \cite{Obl04}. These works provide a variety of embedding criteria and constructions, depending on the structure of the underlying process.

In this paper, we solve the problem of the optimal timing of a sale of a \textit{defaultable asset}, i.e., an asset that might lose all its value. The focus of our stylized model is to complement the diffusion risk model by a default time where default happens at the moment the asset price reaches a lower barrier, which we assume for tractability to be zero. The diffusion risk model is given by a risky price process defined by the SDE
\begin{equation}\label{eq:diff}
	\ud R_t = a(R_t) \, \ud t + \sigma(R_t) \, \ud B_t, \qquad R_0 = r_0 > 0,
\end{equation}
where $B$ is a standard Brownian motion. Ruin occurs at the stopping time 
\[
\tau_0 := \inf\{t\geq 0: R_t\leq 0\},
\]
when the process is absorbed at $0$. Therefore, the stopping strategies considered in this paper must occur before (or at) ruin. Thus,
in addition to asking whether there exists a stopping time $\tau$ such that $R_\tau\sim \mu$, we have to ask further whether such a stopping time can be chosen so that $\tau \leq \tau_0$ almost surely. Thus, at the mathematical core, our results are a novel extension of the Skorokhod embedding problem to stopped diffusions, where the stopping time is of barrier-hitting type.

A complementary interpretation of this problem is within the framework of (actuarial) ruin theory. Initiated by Lundberg and further developed by Cram\'er \cite{Cra30}, ruin theory models the surplus of an insurance company or financial entity over time. In the classical setting, the company begins with an initial surplus, receives premium income continuously, and pays random claims at random
times. Ruin occurs when the surplus becomes non-positive, meaning that accumulated losses have exhausted the initial reserve and incoming premiums. Many risk models have been developed in the ruin-theory literature, including the Cram\'er--Lundberg model, driven by a compound Poisson process \cite{AA10, AS20}; diffusion risk models, in which the surplus follows a diffusion process \cite{Bro95, TM03, YY12}; spectrally negative L\'evy insurance risk models \cite{KKM04}; and jump-diffusion risk models \cite{YZ03}. In this framework our problem can be phrased as follows: Assuming that a financial company wants to sell a portfolio of insurance contracts with surplus process modeled by a diffusion, what is the best time to sell the portfolio and which revenue distributions can be achieved by selling the portfolio?

At the core of this work is the following fundamental conceptual question:
\begin{itemize}
	    \item[\textbf{A.}] Assuming that the price process follows the stochastic dynamics \eqref{eq:diff}, which distributions are attainable by a selling strategy given by (optimal) stopping before ruin?
\end{itemize}
This ruin-constrained embedding problem differs substantially from the classical Skorokhod embedding problem. In the classical setting, the admissible stopping times are not restricted by a random absorbing horizon. In contrast, here the ruin constraint imposes both a state constraint and a random time constraint. As a result, the attainable distributions depend not only on the dynamics of $R$, but also on the ruin boundary. Establishing criteria for attainability under ruin therefore requires new arguments, even for diffusion risk models.

After determining whether a target distribution is attainable, two natural optimization questions arise.
\begin{itemize}
	\item[\textbf{B.}] If a target distribution is not attainable, can one find the closest attainable distribution?
	\item[\textbf{C.}] If a target distribution is attainable, can one find another attainable distribution that is preferable to it? If so, which distribution is optimal?
\end{itemize}

To make these questions precise, we need two additional ingredients: an order relation that compares distributions, and a discrepancy functional that measures the distance between distributions. For the order relation, we use \textit{first-order stochastic dominance}. Let $\mu$ and $\nu$ be probability measures; we say that $\nu$ dominates $\mu$ in the sense of
first-order stochastic dominance, and write $\nu \succeq \mu$, if
\[
    F_\nu(x) \leq F_\mu(x), \qquad \text{ for all }x\in\mathbb{R},
\]
where $F_\mu (x) = \mu\bigl((-\infty,x]\bigr)$ is the cumulative distribution function of $\mu$ (and likewise $F_\nu$ that of $\nu$). We write $\nu \succ \mu$ if $\nu\succeq \mu$ and the inequality is strict for at least one value of $x$. If a probability measure $\mu$ is dominated by an attainable probability measure $\nu$, then we say that $\mu$ is \textit{super-attainable}. This terminology reflects the idea that, although the original target distribution may not itself be attainable, it can be dominated by a distribution that is attainable before ruin. This is analogous to the idea of superhedging, only at the level of distributions instead of random variables.

We need as well a notion of distance between two probability measures to be able to find an optimal measure by minimizing this distance. Therefore we introduce a discrepancy functional between probability distributions and minimize it over an admissible class of distributions. Discrepancy measures between probability distributions have been extensively studied in statistics, information theory, and optimal transport. The functionals used to quantify such discrepancies do not need to be metrics, nor do they need to be symmetric. Once a discrepancy functional is fixed, \textbf{Problems B} and \textbf{C} can be formulated as constrained optimization problems: one minimizes the discrepancy between the original target distribution and an attainable distribution, subject to the ruin and stochastic dominance constraints.

We will focus on a specific class of discrepancy measures, namely extended $f$-divergences. This class includes many popular statistical distances such as the Kullback--Leibler divergence, R\'enyi-type divergences, Hellinger-type distances and the total variation distance \cite{KKK13, KL51, Ren61,Csi67}. In general, extended $f$-divergences come in two flavors: classical $f$-divergences, which only allow candidate distributions that are absolutely continuous with respect to the original target distribution, as well as proper extended $f$-divergences, where candidate distributions may also have singular mass. Classical $f$-divergences do not place mass on events that the original target distribution assigns zero probability to and regard such outcomes as impossible. From the distribution builder perspective, the optimized law retains the set of possible outcomes from the client's original preferences rather than allowing also other outcomes. However, if the selected target distribution is too aggressive or unrealistic in general, the investor might be better served with an investment strategy explicitly allowing bankruptcy or default, even if the investor originally did not consider that. Mathematically, this can be represented by a singular component, for instance a point mass at the default state $0$; such cases are allowed within the extended $f$-divergence framework. The question of minimizing $f$-divergences under linear constraints traces back to Csisz\'{a}r's seminal work \cite{Csi67} and further notable work in this direction is, e.g., \cite{BL91} and \cite{Leo08}. Given the idiosyncratic setting of our problem with potentially unbounded scale function, we do not rely on classical duality theory but prove existence of the optimizer directly in general and provide an analytic subgradient formulation that, while not covering all hypothetical cases, is sufficient for almost all practically important problems. 

The rest of the paper is organized as follows. In Section~\ref{sec:embed}, we introduce the diffusion model with ruin and establish criteria under which distributions are attainable before ruin, thereby solving \tb{Problem}~\textbf{A}. We also discuss the construction and basic properties of the ruin time. In Section~\ref{sec:opt}, we formulate and analyze the constrained optimization problems associated with \tb{Problems}~\textbf{B} and~\textbf{C} within the extended $f$-divergence framework. Section~\ref{sec:ex} applies these general results to specific discrepancy functionals, and presents analytical and numerical examples as well as counterexamples; Section \ref{sec:conc} concludes. Technical results are relegated to the appendix.

\section{Embedding of a measure in a ruined diffusion}
\label{sec:embed}
Let $\bigl(\Omega,\mathcal{F},(\mathcal{F}_t)_{t\geq 0},\mathbb{P}\bigr)$
be a complete filtered probability space satisfying the usual conditions, and let
$B=(B_t)_{t\geq 0}$ be a standard $(\mathcal{F}_t)$-Brownian motion. Let
$I=(\ell,\rho)$ be an open interval such that $-\infty \leq \ell <0<r_0<\rho\leq \infty$. 
We consider a one-dimensional It\^o diffusion
$R=(R_t)_{t\geq 0}$ on $I$ satisfying
\begin{equation}\label{Diff}
    \ud R_t = a(R_t)\,\ud t+\sigma(R_t)\,\ud B_t,
    \qquad R_0=r_0,
\end{equation}
where $a,\sigma:I\to\mathbb{R}$ are Borel measurable functions satisfying the
Engelbert--Schmidt conditions, i.e., $\sigma^2(x)>0$ for all $x\in I$, and $1/\sigma^2$ and $|a|/\sigma^2$ are locally integrable at
every point in $I$.
These assumptions ensure that the diffusion is well defined in the weak sense,
at least up to the first exit time from $I$. Throughout this section, we assume
in addition that $R$ is non-explosive on $I$.

The ruin time is defined by
\[
    \tau_0:=\inf\{t\geq 0:R_t\leq 0\},
\]
and the corresponding ruined process is defined by
\[
    \bar R_t:=R_{t\wedge \tau_0},
    \qquad t\geq 0.
\]
Thus, after ruin occurs, the process is absorbed at the ruin level $0$.

In this section, we study the Skorokhod embedding problem for the ruined
diffusion $\bar R$. More precisely, given a probability measure $\mu$ on
$[0,\rho)$, we seek an almost surely finite stopping time $\tau$ with respect to
$(\mathcal{F}_t)_{t\geq 0}$ such that $\bar R_\tau \sim \mu$.
If such a stopping time exists, we call $\mu$ \textit{attainable} for the ruined
diffusion $\bar R$.

The corresponding embedding problem for the unstopped diffusion $R$, without the
ruin constraint, has been studied in \cite{PP01,CH04}, following the
Az\'ema--Yor approach \cite{AY79}. To recall these results, define the scale
function $S:I\to\mathbb{R}$ by
\begin{equation}\label{eq:scale}
    S(x):=\int_{r_0}^{x} \exp\left\{-2\int_{r_0}^{u}\frac{a(r)}{\sigma^2(r)}\,\ud r \right\} \,\ud u .
\end{equation}
Then $S(r_0)=0$, and $S$ is strictly increasing on $I$. For a probability measure
$\mu$ on $I$ such that $S$ is $\mu$-integrable, define its scaled mean by
\[
    m^\mu := \int_I S(u)\,\mu(\ud u).
\]
The following lemma is Lemma~9 in \cite{CH04}, which extends Theorem~2.1 in
\cite{PP01} and gives a Skorokhod embedding criterion for one-dimensional
diffusion processes.

\begin{lemma}\label{lem:thm2}
Let $\mu$ be a probability measure on $I$ such that $\int_I |S(u)|\,\mu(\ud u)<\infty$.
Then $\mu$ is attainable for the diffusion $R$ if and only if one of the following
conditions holds:
\begin{itemize}
    \item[\textup{(i)}] $S(\ell)=-\infty$ and $S(\rho)=\infty$;
    \item[\textup{(ii)}] $S(\ell)=-\infty$, $S(\rho)<\infty$, and
    $m^\mu\geq 0$;
    \item[\textup{(iii)}] $S(\ell)>-\infty$, $S(\rho)=\infty$, and
    $m^\mu\leq 0$;
    \item[\textup{(iv)}] $S(\ell)>-\infty$, $S(\rho)<\infty$, and
    $m^\mu=0$.
\end{itemize}
Moreover, when any of these conditions holds, an Az\'ema--Yor type stopping time
$T_\mu^{AY}$ can be constructed such that $R_{T_\mu^{AY}}\sim \mu$.

In the case that $\mu$ is a point mass, $\mu = \delta_z$ for $z \in I$, these are classical hitting times,
\begin{equation}\label{Tmu0}
    T_{\mu}^{AY}=\left\{ \begin{array}{ll} \inf \bigl\{t>0 \, : \, R_t\leq z\bigr\}, & z < r_0,\\  0, & z = r_0,\\  \inf \bigl\{t>0 \, : \, R_t\geq z\bigr\}, & z > r_0.\end{array}\right.
\end{equation} 
Otherwise we proceed as follows.

When $m^{\mu}\ge 0$, define
\begin{equation}\label{Tmu1}
    T_{\mu}^{AY}=\inf \bigl\{t>0 \, : \, R_t\leq h_+(M_t)\bigr\},
\end{equation} 
where $M_t:=\sup_{0\leq s\leq t}R_s$ and $x\mapsto h_+(x)$ is an increasing function defined as: $h_+(x)=-\infty$ for $x\leq S^{-1}(m^{\mu})$; $h_+(x)=x$ for $x\ge S^{-1}(\beta_+)$ with $\beta_+=\sup\bigl\{y\in\R:\mu\left([S^{-1}(y),\rho)\right)>0\bigr\}$,
and 
\begin{eqnarray*}
    h^{-1}_+(z)&:=&S^{-1}\biggl(\frac{1}{\mu\left([z,\rho)\right)}\int_{[z,\rho)}S(u)\mu(\ud u)\biggr)
\end{eqnarray*}
for $S^{-1}(m^{\mu})<x<S^{-1}(\beta_+)$.

When $m^{\mu}\leq 0$, define
\begin{equation}\label{Tmu2}
    T_{\mu}^{AY}=\inf\bigl\{t>0: R_t\ge h_-(N_t)\bigr\},
\end{equation} 
where $N_t:=\inf_{0\leq s\leq t}R_s$ and $x\mapsto h_-(x)$ is an increasing function defined as:
$h_-(x)=\infty$ for $x\ge S^{-1}(m^{\mu})$; $h_-(x)=x$ for $x\leq S^{-1}(\beta_-)$ with $\beta_-=\inf\bigl\{y\in\R:\mu\left((\ell,S^{-1}(y)]\right)>0\bigr\}$,
and
\begin{eqnarray*}
    h^{-1}_-(z)&:=&S^{-1}\biggl(\frac{1}{\mu\left((\ell,z]\right)}\int_{(\ell,z]}S(u)\mu(\ud u)\biggr)
\end{eqnarray*}
for $S^{-1}(\beta_-)<x<S^{-1}(m^{\mu})$.
\end{lemma}

To study the embedding problem involving ruin, we first establish the relationship
between the ruin time $\tau_0$ and the Az\'ema--Yor stopping time
$T_\mu^{AY}$ defined in \eqref{Tmu1} or \eqref{Tmu2}.

\begin{proposition}\label{prop:Prop1}
Let $\mu$ be a probability measure supported on $[0,\rho)$ such that the scale function $S$ is $\mu$-integrable. Assume that $\mu$ is attainable for the diffusion $R$, and let $T_\mu^{AY}$ be the Az\'ema--Yor stopping time constructed in Lemma~\ref{lem:thm2}. Then the following statements hold.
	\begin{itemize}
\item[\textup{(i)}] If $m^\mu=0$, then $\mathbb{P}\bigl(\tau_0<T_\mu^{AY}\bigr)=0$.
\item[\textup{(ii)}] If $m^\mu>0$ and $S(\ell)=-\infty$, then $ \mathbb{P}\bigl(\tau_0<T_\mu^{AY}\bigr)\leq\frac{m^\mu}{m^\mu-S(0)}$.
\item[\textup{(iii)}] If $m^\mu<0$ and $S(\rho)=\infty$, then $\mathbb{P}\bigl(\tau_0<T_\mu^{AY}\bigr)=0$.
\end{itemize}
\end{proposition}

\begin{proof}
First assume that $m^\mu>0$ and $S(\ell)=-\infty$ in which case the Az\'ema--Yor stopping time $T_\mu^{AY}$ is given by \eqref{Tmu1}. On the event $\{\tau_0<T_\mu^{AY}\}$, the process has not yet
crossed the Az\'ema--Yor upper barrier $h_+$ at time $\tau_0$. Since $R_{\tau_0}=0$, we have
\[
    \bigl\{\tau_0<T_\mu^{AY}\bigr\} \subseteq \bigl\{0>h_+(M_{\tau_0})\bigr\}.
\]
By the definition of $h_+$ and the assumption $\supp\mu\subseteq [0,\rho)$,
\[
    h_+^{-1}(0)
    =
    S^{-1}\left(
        \frac{1}{\mu([0,\rho))}
        \int_{[0,\rho)}S(u)\,\mu(\ud u)
    \right)
    =
    S^{-1}(m^\mu).
\]
Since $m^\mu>0=S(r_0)$ and $S$ is strictly increasing, we have $S^{-1}(m^\mu)>r_0$. Therefore,
\[
    \bigl\{0>h_+(M_{\tau_0})\bigr\}
    \subseteq
    \bigl\{M_{\tau_0}<S^{-1}(m^\mu)\bigr\}.
\]
Let $H_y:=\inf\{t\geq 0:R_t=y\}$, then
\[
    \bigl\{M_{\tau_0}<S^{-1}(m^\mu)\bigr\}
    \subseteq
    \bigl\{\tau_0 < H_{S^{-1}(m^\mu)}\bigr\}.
\]
Using the standard two-sided hitting probability for one-dimensional diffusions, we obtain
\[
    \mathbb{P}\bigl(\tau_0 < H_{S^{-1}(m^\mu)}\bigr)
    =
    \frac{
        S(S^{-1}(m^\mu))-S(r_0)
    }{
        S(S^{-1}(m^\mu))-S(0)
    }.
\]
Since $S(r_0)=0$, this gives
\[
    \mathbb{P}\bigl(\tau_0<T_\mu^{AY}\bigr)
    \leq
    \frac{m^\mu}{m^\mu-S(0)}.
\]

Next assume that $m^\mu=0$. Again using the upper-barrier construction \eqref{Tmu1}, we have
\[
    h_+^{-1}(0)=S^{-1}(m^\mu)=S^{-1}(0)=r_0.
\]
Hence
\[
    \bigl\{\tau_0<T_\mu^{AY}\bigr\}
    \subseteq
    \bigl\{0>h_+(M_{\tau_0})\bigr\}
    \subseteq
    \bigl\{M_{\tau_0}<r_0\bigr\}.
\]
However, $M_{\tau_0}\geq R_0=r_0$ almost surely. Therefore, $\mathbb{P}\bigl(\tau_0<T_\mu^{AY}\bigr)=0$.

Finally assume that $m^\mu<0$ and $S(\rho)=\infty$. Then the Az\'ema--Yor stopping time is given by \eqref{Tmu2}. On the event $\{\tau_0<T_\mu^{AY}\}$, the process has not yet crossed the lower barrier $h_-$ at time $\tau_0$. Since $R_{\tau_0}=0$ and $N_{\tau_0}=0$, we obtain
\[
    \bigl\{\tau_0<T_\mu^{AY}\bigr\}
    \subseteq
    \bigl\{0<h_-(N_{\tau_0})\bigr\}
    =
    \bigl\{0<h_-(0)\bigr\}.
\]
By the definition of the lower Az\'ema--Yor barrier and the support condition $\supp\mu\subseteq [0,\rho)$, the lower endpoint of the support is not below $0$. Hence $h_-(0)=0$.
Consequently,
\[
    \bigl\{\tau_0<T_\mu^{AY}\bigr\}\subseteq \{0<0\}=\emptyset,
\]
and therefore $\mathbb{P}\bigl(\tau_0<T_\mu^{AY}\bigr)=0$.
\end{proof}

We now state our first main result, which gives a criterion for determining which distributions are attainable by the ruined diffusion. This solves Problem~\textbf{A} stated in the introduction.
\begin{theorem}\label{thm:mainthm}
Let $\mu$ be a probability measure  supported on $[0,\rho)$ such that the scale function $S$ is $\mu$-integrable. Then $\mu$ is attainable for the ruined diffusion $\bar R$ if and only if
\[
    m^\mu=0,
    \qquad\text{or}\qquad
    m^\mu<0 \ \text{ and }\ S(\rho)=\infty .
\]
Equivalently, $\mu$ is attainable for $\bar R$ if and only if $m^\mu\leq 0$ and, whenever $m^\mu<0$, one has $S(\rho)=\infty$.
\end{theorem}

\begin{proof}
We first prove the only if part: Suppose that $\mu$ is attainable for the ruined diffusion $\bar R$. Then there exists an almost surely finite stopping time
$T_\mu$ such that $\bar R_{T_\mu}=R_{T_\mu\wedge \tau_0}\sim \mu$.
In particular, $\mu$ is also attainable for the original diffusion $R$, with
embedding time $T_\mu\wedge \tau_0$. 

For $t\geq 0$, define $Y_t:=S(R_t)$. By the definition of the scale function, $Y$ is a local martingale. Therefore
$ Y^{\tau_0}_\cdot:=Y_{\cdot\wedge \tau_0}=S(R_{\cdot\wedge \tau_0})$ is also a local martingale. Since $R_{t\wedge \tau_0}\geq 0$ and $S$ is increasing, we have $Y^{\tau_0}_t\geq S(0)$ for all $t\geq 0$. Thus $Y^{\tau_0}$ is a supermartingale. Applying optional sampling to the bounded-below supermartingale $Y^{\tau_0}$ at the almost surely finite stopping time $T_\mu$, we obtain
\[
    m^\mu=
    \int_{[0,\rho)}S(u)\,\mu(\ud u)= \mathbb{E}\bigl[S(R_{T_\mu\wedge \tau_0})\bigr]
    = \mathbb{E}\bigl[Y^{\tau_0}_{T_\mu}\bigr]
    \leq \mathbb{E}\bigl[Y^{\tau_0}_0\bigr]
    = S(r_0) = 0 .
\]
In addition, if $m^\mu<0$, then the attainability of $\mu$ for the original diffusion $R$, together with Lemma~\ref{lem:thm2}, implies that necessarily $S(\rho)=\infty$. 

To prove the if part, we first suppose that $m^\mu=0$. By Lemma~\ref{lem:thm2}, there exists an Az\'ema--Yor stopping time $T_\mu^{AY}$ such that $R_{T_\mu^{AY}}\sim \mu$. Moreover, by Proposition~\ref{prop:Prop1}, $\mathbb{P}\bigl(\tau_0<T_\mu^{AY}\bigr)=0$. Hence $T_\mu^{AY}\leq \tau_0$ almost surely, and therefore $\bar R_{T_\mu^{AY}}=R_{T_\mu^{AY}\wedge \tau_0}=R_{T_\mu^{AY}} \sim \mu$. Thus $\mu$ is attainable for the ruined diffusion $\bar R$.

Next suppose that $m^\mu<0$ and $S(\rho)=\infty$. Since $\supp\mu\subseteq [0,\rho)$ and $S$ is increasing, we have $m^\mu\geq S(0)$.
If $m^\mu=S(0)$, then necessarily $\mu=\delta_0$, the Dirac measure at $0$. In this case, $\bar R_{\tau_0}=0$
on $\{\tau_0<\infty\}$. Define $H_y:=\inf\{t\geq 0:R_t=y\}$; since $S(\rho)=\infty$, the standard two-sided hitting
probability gives
\[
\mathbb{P}(\tau_0<\infty) = \lim_{b\uparrow\rho} \mathbb{P}(H_0<H_b)=\lim_{b\uparrow\rho}\frac{S(b)-S(r_0)}{S(b)-S(0)}
=1.
\]
Thus $\tau_0$ embeds $\delta_0$ in the ruined diffusion.

It remains to consider the case $S(0)<m^\mu<0$. Then, $0<S^{-1}(m^\mu)<r_0$. Here we will proceed in two steps: First we aim to reach the constant level $S^{-1}(m^\mu)>0$, then we aim from there for the target distribution. Using again the two-sided hitting probability and the assumption $S(\rho)=\infty$, we obtain
\[
\mathbb{P}(H_{S^{-1}(m^\mu)}<\infty)=\lim_{b\uparrow\rho}\mathbb{P}(H_{S^{-1}(m^\mu)}<H_b)=\lim_{b\uparrow\rho}\frac{S(b)-S(r_0)}{S(b)-S(S^{-1}(m^\mu))}=1.
\]
Moreover, since $0<S^{-1}(m^\mu)<r_0$ and the sample paths are continuous, $H_{S^{-1}(m^\mu)}<\tau_0$ almost surely.

Now, we restart the original process $R$ from the lower level $S^{-1}(m^\mu)$, denoted by $\widetilde R$. For this restarted diffusion, define the shifted scale function
\[
\widetilde S(x):=S(x)-S\bigl(S^{-1}(m^\mu)\bigr)=S(x)-m^\mu .
\]
Then $\widetilde S\bigl(S^{-1}(m^\mu)\bigr)=0$ and
\[
\widetilde m^\mu:=\int_{[0,\rho)}\widetilde S(u)\,\mu(\ud u)=\int_{[0,\rho)}\bigl(S(u)-m^\mu\bigr)\,\mu(\ud u)=m^\mu-m^\mu=0 .
\]
By the already proved case $m=0$, applied to the diffusion started from $S^{-1}(m^\mu)$, there exists an almost surely finite stopping time $\widetilde T_\mu^{AY}\leq \widetilde \tau_0$ such that $\widetilde{\bar R}_{\widetilde T_\mu^{AY}}=\widetilde R_{\widetilde T_\mu^{AY}}\sim \mu$, where $\widetilde \tau_0:=\inf\{t\geq 0:\widetilde R_t\leq 0\}$.

Now, define the stopping time for the original process by restarting this construction
at $H_{S^{-1}(m^\mu)}$:
\[
    T_\mu:=H_{S^{-1}(m^\mu)}+\widetilde T_\mu^{AY}\circ\theta_{H_{S^{-1}(m^\mu)}},
\]
where $\theta$ denotes the usual shift operator. By the strong Markov property
of $R$ at $H_{S^{-1}(m^\mu)}$, the process $(R_{H_{S^{-1}(m^\mu)}+t})_{t\geq 0}$ has the same law as the diffusion started from $S^{-1}(m^\mu)$, and is independent of the term $R_{H_{S^{-1}(m^\mu)}}=S^{-1}(m^\mu)$ on which we conditioned. Therefore, for every bounded Borel function $\varphi$,
\[
\mathbb{E}^{r_0}\bigl[\varphi(R_{T_\mu})\bigr]=    \mathbb{E}^{S^{-1}(m^\mu)}\bigl[\varphi(\widetilde{R}_{\widetilde T_\mu^{AY}})\bigr]=\int_{[0,\rho)} \varphi(u)\,\mu(\ud u).
\]
Thus $R_{T_\mu}\sim \mu$.
Since $H_{S^{-1}(m^\mu)}<\tau_0$ almost surely and the shifted stopping time satisfies
\[
    \widetilde T_\mu^{AY}\circ\theta_{H_{S^{-1}(m^\mu)}}
    \leq
    \tau_0-H_{S^{-1}(m^\mu)}
    \qquad \text{a.s.},
\]
we have $T_\mu\leq \tau_0$ almost surely. Consequently,
\[
\bar R_{T_\mu}=R_{T_\mu\wedge \tau_0}=R_{T_\mu}\sim \mu .
\]
Therefore $\mu$ is attainable for the ruined diffusion $\bar R$.
\end{proof}

The next proposition discusses the remaining case in which $S(\rho)<\infty$. Combining Lemma~\ref{lem:thm2} with Theorem~\ref{thm:mainthm}, we see that, if $S(\rho)<\infty$, then a probability measure $\mu$ with $\supp\mu\subseteq[0,\rho)$ is attainable for the ruined diffusion if and only if $m^\mu=0$. However, if $m^\mu<0$, then $\mu$ can still be dominated by an attainable distribution.

\begin{proposition}\label{prop:sa}
Let $\mu$ be a probability measure supported on $[0,\rho)$ such that the scale function $S$ is $\mu$-integrable.
Assume that $S(\rho)<\infty$ and $m^\mu<0$. Then $\mu$ is super-attainable for the ruined diffusion $\bar R$.
\end{proposition}

\begin{proof}
Let $F^\mu$ be the cumulative distribution function (CDF) of $\mu$. We define a new distribution $\nu$ with its CDF
\[
F^\nu(x)
    =
    \begin{cases}
        cF^\mu(x), & x<r_0\\[4pt]
        F^\mu(x), & x\geq r_0
    \end{cases},\ \ \mbox{with}\ \ c= - \frac{ \int_{[r_0,\rho)} S(u)\mu(\ud u)}{\int_{[0,r_0)}S(u)\mu(\ud u)}.
\]
Since $m^\mu<0$, $S(r_0)=0$, and $S$ is increasing, we have
\[
0\leq\int_{[r_0,\rho)} S(u)\,\mu(\ud u)< -\int_{[0,r_0)} S(u)\,\mu(\ud u)<\infty,
\]
whence $c\in[0,1)$. Therefore, $F^\nu(x)\leq F^\mu(x)$ for all  $x\in\mathbb{R}$, which implies $\nu\succeq\mu$. Moreover, since $S(r_0)=0$, $m^\nu=c\int_{[0,r_0)}S(u)\,\mu(\ud u)+\int_{[r_0,\rho)}S(u)\,\mu(\ud u)=0$.
Also, $\int_{[0,\rho)}|S(u)|\,\nu(\ud u)\leq\int_{[0,\rho)}|S(u)|\,\mu(\ud u)<\infty$.
Thus, by Theorem~\ref{thm:mainthm}, $\nu$ is attainable for the ruined diffusion
$\bar R$, i.e., $\mu$ is super-attainable.
\end{proof}

Theorem~\ref{thm:mainthm} suggests a modified Az\'ema--Yor stopping-time construction: one first waits until the diffusion hits the level $S^{-1}(m^\mu)$ and then applies an Az\'ema--Yor type stopping time from that level. When $S(0)<m^\mu\leq 0$, we have $0\leq S^{-1}(m^\mu)\leq r_0$.
Thus, by continuity of the sample paths, the hitting time of $S^{-1}(m^\mu)$ always occurs before ruin. This allows us to restart the diffusion at $S^{-1}(m^\mu)$, shift the scale function so that the scaled mean becomes zero, and then apply Lemma~\ref{lem:thm2}.

If the assumption $S(\rho)=\infty$ is removed, then the hitting time $H_{S^{-1}(m^\mu)}$ need not be finite almost surely. In that case, the modified Az\'ema--Yor construction above may fail. Proposition~\ref{prop:sa} addresses this obstruction in the case $S(\rho)<\infty$: although a distribution $\mu$ with $m^\mu<0$ is not attainable for the ruined diffusion, one can construct another distribution $\nu$ such that $\nu\succeq\mu$ and $m^\nu=0$, which implies that $\mu$ is super-attainable.

\begin{remark}\label{rem:compCS}
Comparing the results in this section with those obtained by Carr and Sturm \cite[Section 4]{CS26} for the case without ruin, we see that the difference is for distributions with positive scaled mean $m^\mu$ in the case that $S(\ell) = -\infty$. Without ruin, they can be attained, while in the case with ruin they cannot, as discussed in Theorem \ref{thm:mainthm}.
\end{remark}

\section{Optimality under extended \texorpdfstring{$f$}{f}-divergence}
\label{sec:opt}
In the previous section, we have shown which probability measures are (super-)attainable and found that the sign of the scaled mean plays a fundamental role in the attainability and super-attainability of a target distribution. However, this leaves open two fundamental questions: If the distribution chosen by an investor is not super-attainable, is there a way for them to find a similar distribution that is attainable? Of course, one could just provide an output of the scaled mean and use it as an indicator similar to the "budget meter" in the original distribution builder \cite{SGB00}. A more robust approach than fidgeting around would be to determine the mathematically closest distribution that is super-attainable, which is \textbf{Problem B}. Similarly, it might be that the chosen target distribution is super-attainable but not optimal, as there are other distributions that are strictly dominating the chosen distribution in first order stochastic dominance and are still attainable, leading to the constrained optimization problem \textbf{Problem C}. To tackle these problems mathematically, one first has to determine what "the closest distribution" actually means --- there are many notions of distance between probability measures without a single one being the obviously correct choice. For the current paper we consider the discrepancy between probability measures being measured by an extended $f$-divergence; see, e.g., \cite{Csi67}. This is a very rich class of statistical distances that encompasses, among others, the classical Kullback--Leibler and R\'{e}nyi divergences as well as the Hellinger and Bhattacharyya distances and the total variation distance.

In a first step, we make our problem formulation mathematically precise. Throughout this section, $\mu$ is taken to be a fixed probability measure supported on $[0,\rho)$. Whenever there is no risk of confusion, notation depending on $\mu$ will be written without explicitly indicating this dependence. 

Let $\mathcal P$ denote the set of all probability measures on $[0,\rho)$ satisfying the scaled-integrability condition,
\[
    \mathcal P:=\left\{\nu:\ \int_{[0,\rho)} |S(u)|\nu(\ud u)<\infty\right\}.
\]

For any probability measures $\mu, \nu\in\mathcal P$, recall that the Lebesgue--Radon--Nikod\'{y}m theorem provides a unique decomposition of $\nu$ with respect to $\mu$ as
\[
\nu=\nu^{ac}+\nu^\perp,\qquad\nu^{ac}\ll\mu,\qquad\nu^\perp\perp\mu,
\]
where $\nu^{ac}$ denotes the absolutely continuous part of $\nu$ with respect to $\mu$, and $\nu^\perp$ denotes the singular part of $\nu$ with respect to $\mu$. Moreover, there exists a unique nonnegative $\theta \in L^1 :=L^1\bigl([0, \rho), \mathcal{B}\bigl([0, \rho)\bigr), \mu\bigr)$, the space of $\mu$-integrable functions, such that
\[
\nu^{ac}(A) = \int_{[0,\rho)} \theta(u) \ind_A(u) \, \mu(\ud u) \qquad \text{for all } A \in \mathcal{B}\bigl([0, \rho)\bigr).
\]

Let $f:(0,\infty)\to(-\infty,\infty)$ be a convex function extended to $(-\infty,\infty]$ by setting $f(0)=f(0+):=\lim_{x\downarrow 0}f(x)$ and $f(x) = \infty$ for $x<0$. This extended function is lower semicontinuous on its entire domain as it is continuous on $(0,\infty)$. Its recession constant is well defined as
\[
    f'(\infty):=\lim_{x\to\infty}\frac{f(x)}{x}
    \in(-\infty,\infty].
\]

For the remainder of this section, we will adopt the following assumption.
\begin{ass}
\label{ass}
Without loss of generality, throughout this section we assume that $f\geq 0$, $f(1)=0$, and $f'(\infty)\geq 0$.
\end{ass}
Indeed, by convexity, there exists $\gamma\in\mathbb{R}$ such that
$f(x)\geq f(1)+\gamma(x-1)$ for all $x\geq 0$. Defining
\[
    \widetilde f(x):=f(x)-\gamma(x-1)-f(1),
\]
we have $\widetilde f$ is convex, $\widetilde f\geq 0$ , $\widetilde f(1)=0$, and $\widetilde f'(\infty)=f'(\infty)-\gamma\geq 0$. As optima of extended $f$-divergences, defined as follows, are invariant under affine transformations of $f$, the assumption indeed does not lessen generality.

\begin{definition}\label{def:f-div}
The extended $f$-divergence of $\nu$ with respect to $\mu$ is defined by
\begin{equation}\label{fdiv}
    D_f(\nu \, \Vert \, \mu)
    :=
    \int_{[0,\rho)} f\bigl(\theta(u)\bigr)\,\mu(\ud u)
    +f'(\infty)\nu^\perp\bigl([0,\rho)\bigr).
\end{equation}
When $f'(\infty)=\infty$, we adopt the usual convention that
$a\cdot\infty=0$ if $a=0$, and $a\cdot\infty=\infty$ if $a>0$.
\end{definition}

The notation $D_f(\nu \, \Vert \, \mu)$ emphasizes that the divergence is generally
not symmetric in its two arguments.

To make our problems precise, we provide the following definitions.

\noindent\textbf{Problem B.}
Given $\mu$ with $m^\mu>0$, solve
\begin{equation}\label{eq:B}
    \inf_{\nu\in\mathcal P}D_f(\nu \, \Vert \, \mu),
    \qquad
    \text{subject to}
    \qquad
    \int_{[0,\rho)}S(u)\,\nu(\ud u) \leq 0.
    \tag{B}
\end{equation}

\bigskip

\noindent\textbf{Problem C.}
Given $\mu$ with $m^\mu<0$, solve
\begin{equation}\label{eq:C}
\left\{\begin{aligned}
    \inf_{\nu\in\mathcal P}D_f(\nu \, \Vert \, \mu),
    \qquad
    &\text{subject to}
    \qquad
    \int_{[0,\rho)}S(u)\,\nu(\ud u)\leq 0,
    \qquad
    \nu\succeq\mu,\\
    &\text{and there exists no } \nu' \in \mathcal{P} \text{ such that } \\
    &\int_{[0,\rho)}S(u)\,\nu'(\ud u)\leq 0, \qquad
    \nu'\succ \nu.
    \end{aligned}\right.
    \tag{C}
\end{equation}

Since \textbf{Problem C} may be infeasible in settings where mass can escape to infinity, we also consider a compact support version. For some $M \geq \sup \supp \mu$, we define
\begin{equation}\label{eq:CM}
\mathcal{P}^M := \bigl\{ \nu \in \mathcal{P} \, : \, \supp \nu \subseteq [0,M]\bigr\}.\tag{CM}
\end{equation}
Replacing $\mathcal P$ by $\mathcal P^M$ in \textbf{Problem C}, we obtain \textbf{Problem CM}.

For practical reasons we prefer a different formulation of the problems that stem from the fact that the optimizers are necessarily realized for zero scaled mean: 

\bigskip
\noindent\textbf{Problem B'.}
Given $\mu$ with $m^\mu>0$, solve
\begin{equation}\label{eq:B'}
    \inf_{\nu\in\mathcal P}D_f(\nu \, \Vert \, \mu),
    \qquad
    \text{subject to}
    \qquad
    \int_{[0,\rho)}S(u)\,\nu(\ud u) = 0.
    \tag{B'}
\end{equation}

\bigskip

\noindent\textbf{Problem C'} (resp. \textbf{CM'}).
Given $\mu$ with $m^\mu<0$, solve
\begin{equation}\label{eq:C'}
    \inf_{\nu\in\mathcal P}D_f(\nu \, \Vert \, \mu),
    \qquad
    \text{subject to}
    \qquad
    \int_{[0,\rho)}S(u)\,\nu(\ud u) = 0, \qquad
    \nu\succeq\mu.
    \tag{C'}
\end{equation}
Resp. given $\mu$ with $\supp \mu \subseteq [0,M]$ and $m^\mu<0$, solve
\begin{equation}\label{eq:CM'}
    \inf_{\nu\in\mathcal{P}^M}D_f(\nu \, \Vert \, \mu),
    \qquad
    \text{subject to}
    \qquad
    \int_{[0,\rho)}S(u)\,\nu(\ud u) = 0, \qquad
    \nu\succeq\mu.
    \tag{CM'}
\end{equation}

The following proposition shows that, in the remainder of this section, it suffices to work with problems \eqref{eq:B'}, \eqref{eq:C'}, and \eqref{eq:CM'} when studying problems, \eqref{eq:B}, \eqref{eq:C}, and \eqref{eq:CM}, respectively.
\begin{prop}\label{prop:equiv}
Problems \eqref{eq:B}, \eqref{eq:B'}, \eqref{eq:C}, \eqref{eq:C'}, \eqref{eq:CM}, and \eqref{eq:CM'} are related as follows:
\begin{itemize} 
\item[\textup{(i)}] Problems \eqref{eq:B} and \eqref{eq:B'} have the same optimal value, and one admits an optimizer if and only if the other does. Moreover, every optimizer of problem \eqref{eq:B'} is also an optimizer of problem \eqref{eq:B}.
\item[\textup{(ii)}] The feasible sets of problems \eqref{eq:C} and \eqref{eq:C'} agree. 
\item[\textup{(iii)}] For $M\geq r_0$, the feasible sets of problems \eqref{eq:CM} and \eqref{eq:CM'} agree.
\end{itemize}
\end{prop}
\begin{proof}We first prove (i). Since the feasible set of
\eqref{eq:B'} is contained in that of \eqref{eq:B}, the optimal
value of \eqref{eq:B} is no larger than that of  \eqref{eq:B'}.

Conversely, let $\nu\in\mathcal P$ be feasible for  \eqref{eq:B}. If
$m^\nu=0$, then $\nu$ is also feasible for  \eqref{eq:B'}. Suppose that
$m^\nu<0$. For $t\in(0,1)$, define
\[
    \nu_t:=(1-t)\nu+t\mu.
\]
Then
\[m^{\nu_t}=(1-t)m^\nu+t\,m^\mu.
\]
Since $m^\nu<0$ and $m^\mu>0$, choosing
\[t^* :=\frac{-m^\nu}{-m^\nu+m^\mu}\in(0,1)
\]
gives $m^{\nu_{t^*}}=0$.
Hence $\nu_{t^*}$ is feasible for  \eqref{eq:B'}. By the convexity of the
extended $f$-divergence and the fact that $D_f(\mu \, \Vert \, \mu)=0$,
\[
\begin{aligned}
    D_f(\nu_{t^*} \, \Vert \, \mu)
    &\leq
    (1-t^*)D_f(\nu \, \Vert \, \mu)
    +t^*D_f(\mu \, \Vert \, \mu)\\
    &=
    (1-t^*)D_f(\nu \, \Vert \, \mu)
    \leq
    D_f(\nu \, \Vert \, \mu).
\end{aligned}
\]
Thus, for every feasible measure of  \eqref{eq:B}, there exists a feasible
measure of \eqref{eq:B'} with no larger objective value. It follows that
 \eqref{eq:B} and \eqref{eq:B'} have the same optimal value.

Since every feasible measure for \eqref{eq:B'} is also feasible for
 \eqref{eq:B}, if  \eqref{eq:B'} admits an optimizer,
then so does \eqref{eq:B}.

Conversely, suppose that \eqref{eq:B} admits an optimizer $\nu^*$. If
$m^{\nu^*}=0$, then $\nu^*$ is feasible for  \eqref{eq:B'} and, since the
two problems have the same optimal value, it is also an optimizer of
\eqref{eq:B'}. If $m^{\nu^*}<0$, applying the above construction to
$\nu^*$ yields a measure $\nu^*_{t^*}$ feasible for \eqref{eq:B'} such that
\[
D_f(\nu^*_{t^*} \, \Vert \, \mu)\leq D_f(\nu^* \, \Vert \, \mu).
\]
Since the two problems have the same optimal value and $\nu^*$ is optimal for
 \eqref{eq:B}, equality must hold, and $\nu^*_{t^*}$ is an optimizer of
 \eqref{eq:B'}. Therefore, problem \eqref{eq:B} admits an optimizer if and
only if problem \eqref{eq:B'} does.

Next, we prove (ii). Suppose that $\nu$ is feasible for problem \eqref{eq:C'}. Then, 
we claim that there exists no $\nu'\in\mathcal P$
satisfying $m^{\nu'}\leq0$, $\nu'\succeq\mu$, and $\nu'\succ\nu$. Indeed, since $S$ is strictly increasing,
\[
    \nu'\succ\nu
    \quad\Longrightarrow\quad
    m^{\nu'}>m^\nu=0,
\]
which contradicts $m^{\nu'}\leq0$. Hence $\nu$ is feasible for
problem \eqref{eq:C}.

Conversely, suppose that $\nu$ is feasible for problem \eqref{eq:C}.
We show that necessarily $m^\nu=0$, i.e., $\nu$ is feasible for problem \eqref{eq:C'}. Assume, to the contrary, that
$m^\nu<0$. Since $S(r_0)=0$ and $S$ is strictly increasing, $S(u)\geq0$ for
$u\geq r_0$. Hence $m^\nu<0$ implies $\nu\bigl([0,r_0)\bigr)>0$. Then, for $\varepsilon\in(0,1)$, define
\[ \nu_\varepsilon:=\nu-\varepsilon\,\nu|_{[0,r_0)}+\varepsilon\,\nu\bigl([0,r_0)\bigr)\,\delta_{r_0}.
\]
Thus, $\nu_\varepsilon\in\mathcal P$ is obtained from $\nu$ by moving an
$\varepsilon$-fraction of all mass below $r_0$ to $r_0$. Therefore, $\nu_\varepsilon\succ\nu$, which implies $\nu_\varepsilon\succeq\mu$. Moreover, since $S(r_0)=0$,
\[m^{\nu_\varepsilon}=m^\nu-\varepsilon\int_{[0,r_0)}S(u)\,\nu(\ud u).
\]
Since $\int_{[0,r_0)}S(u)\,\nu(\ud u)<0$,
we have $m^{\nu_\varepsilon}\downarrow m^\nu<0$ as $\varepsilon\downarrow0$. 
Hence, for all sufficiently small $\varepsilon>0$, $m^{\nu_\varepsilon}<0$. Thus $\nu_\varepsilon$ satisfies
$m^{\nu_\varepsilon}\leq0$, $\nu_\varepsilon\succeq\mu$, and $\nu_\varepsilon\succ\nu$,
contradicting the feasibility of $\nu$ for problem \eqref{eq:C}.
Therefore $m^\nu=0$.

The compact support case (iii) follows by the same argument. 
\end{proof}

\begin{remark}\label{rem:compCS2}
We note that, while we consider the problem in the context of the ruined
diffusion model of Section~\ref{sec:embed}, the results apply equally to the
model without ruin. All the following analysis depends only on the
attainability properties of the diffusion. The attainability conditions for
the models with and without ruin are the same when $S(\ell)>-\infty$, so the
results apply indiscriminately. If $S(\ell)=-\infty$, then, without ruin,
distributions with positive scaled mean $m^\mu$ are also attainable. In this
case, Problem~\eqref{eq:C'} (resp.\ \eqref{eq:C}) is ill-posed, since for
every target distribution we can find a strictly preferable one. Problem~
\eqref{eq:CM'} (resp.\ \eqref{eq:CM}), on the other hand, is trivially solved
by the point mass $\delta_M$.
\end{remark}

To study existence and uniqueness of solutions to the problems~\eqref{eq:B'}, \eqref{eq:C'} and~\eqref{eq:CM'}, we first introduce the following auxiliary constrained optimization problems, which can be viewed as generalizations of Problem~\eqref{eq:B'} without the restriction on the sign of $m^\mu$.
Define
\[
\mathcal P_0:=\left\{\nu\in\mathcal P:\int_{[0,\rho)}S(u)\,\nu(\ud u)=0\right\}.
\]

\noindent\textbf{Problem G.}
Given $\mu\in\mathcal P$, solve
\begin{equation}\label{eq:G}
    \inf_{\nu\in\mathcal P_0}D_f(\nu \, \Vert \, \mu).
    \tag{G}
\end{equation}

\begin{theorem}\label{thm:eu}
Assume in each case that the corresponding feasible set contains at
least one measure with finite $f$-divergence. Then the following conclusions hold.
\begin{itemize}
\item[\textup{(i)}]
If $m^\mu>0$, then problem~\eqref{eq:G} admits a minimizer, which further implies problem \eqref{eq:B} admits a minimizer.

\item[\textup{(ii)}]
If $m^\mu<0$ and $f'(\infty)=\infty$, then
problem~\eqref{eq:C} admits a minimizer provided that there exists a
constant $c<\infty$ such that
\begin{equation}\label{eq:int}
\lim_{K\to\infty}
\sup_{\substack{\nu\in\mathcal P_0,\nu\succeq\mu\\
D_f(\nu\Vert\mu)\leq c}}
\int_{\{S>K\}}S(u)\,\nu(\ud u)
=0.
\end{equation}
In particular, this condition is automatically satisfied when
$S(\rho)<\infty$.

\item[\textup{(iii)}]
If $m^\mu<0$, then problem~\eqref{eq:CM} admits a minimizer for every
$M\in[r_0,\rho)$ such that $\supp\mu\subseteq[0,M]$.
\end{itemize}

In each case, if $f$ is strictly convex, then the minimizer is unique.
\end{theorem}
\begin{proof} We first prove \textup{(i)}. Let $(\nu_n)_{n\geq1}$ be a minimizing sequence in $\mathcal P_0$. Since the feasible set contains at least one measure with finite $f$-divergence, we may assume that $D_f(\nu_n \, \Vert \, \mu)\leq c$ for some $c<\infty$ independent of $n$. We consider two cases. 

\medskip \noindent\emph{Case 1: $S(\rho)=\infty$.} 
Define $A(x):=S(x)-S(0)$. 
Then $A(x)\geq0$ and $A(x)\to\infty$ as $x\to\rho$. Since each $\nu_n$ is feasible, 
\[ \int_{[0,\rho)}A(u)\,\nu_n(\ud u)=-S(0). 
\] 
Therefore, for every $K>0$, 
\[ \nu_n\bigl(\{A>K\}\bigr) \leq \frac{1}{K}\int_{\{A>K\}}A\,\ud\nu_n \leq \frac{-S(0)}{K}. 
\] 
Hence $(\nu_n)$ is tight. Passing to a subsequence, we may assume that $\nu_n\Rightarrow\nu$. 
By lower semicontinuity, 
\[ D_f(\nu \, \Vert \, \mu) \leq \liminf_{n\to\infty}D_f(\nu_n \, \Vert \, \mu). 
\] 
Moreover, the Portmanteau theorem gives 
\[\int_{[0,\rho)}S(u)\,\nu(\ud u) = S(0)+\int_{[0,\rho)}A(u)\,\nu(\ud u)\leq S(0)+\liminf_{n\to\infty} \int_{[0,\rho)}A(u)\,\nu_n(\ud u) =0. 
\]

 \medskip 
 \noindent\emph{Case 2: $S(\rho)<\infty$.} Extend $S$ continuously to the compactified state space $[0,\rho]$ by setting $S(\rho):=\lim_{x\uparrow\rho}S(x)$. After passing to a subsequence, we may assume that $\nu_n$ converges weakly to a probability measure $\bar\nu$ on $[0,\rho]$. We also regard $\mu$ as a probability measure on $[0,\rho]$ by setting
$\mu\bigl(\{\rho\}\bigr)=0$, and use the same definition of the extended
$f$-divergence on this compactified space. By lower semicontinuity, 
\[ D_f(\bar\nu \, \Vert \, \mu) \leq \liminf_{n\to\infty}D_f(\nu_n \, \Vert \, \mu). 
\] 
Since $S$ is bounded and continuous on $[0,\rho]$, 
\[ \int_{[0,\rho]}S(u)\,\bar\nu(\ud u)=0. 
\] 
Let $d:=\bar\nu\bigl(\{\rho\}\bigr)$. If $f'(\infty)=\infty$, the finiteness of $D_f(\bar\nu \, \Vert \, \mu)$ implies that $d=0$. If $f'(\infty)<\infty$, following Lemma \ref{lem:0}, move the endpoint mass from $\rho$ to $0$ by defining 
\[ \widehat\nu := \bar\nu|_{[0,\rho)}+d\delta_0. 
\] 
We claim that
\[D_f(\widehat\nu \, \Vert \, \mu)\leq D_f(\bar\nu \, \Vert \, \mu).
\]
Indeed, if $\mu\bigl(\{0\}\bigr)=0$, then both $d\delta_\rho$ and
$d\delta_0$ are singular with respect to $\mu$. Hence $D_f(\widehat\nu \, \Vert \, \mu)=D_f(\bar\nu \, \Vert \, \mu)$.

If instead $p_0:=\mu\bigl(\{0\}\bigr)>0$, write $\bar\nu^{ac}=\bar\theta\mu$.
After moving $d\delta_\rho$ to $0$, the density at $0$ changes from
$\bar\theta(0)$ to
\[
    \bar\theta(0)+\frac{d}{p_0}.
\]
Since $d\delta_\rho$ contributes the singular cost $\alpha d$ to
$D_f(\bar\nu \, \Vert \, \mu)$, we obtain, with $\alpha= f'(\infty)$,
\[ D_f(\widehat\nu \, \Vert \, \mu)-D_f(\bar\nu \, \Vert \, \mu)=p_0\left[f\left(\bar\theta(0)+\frac{d}{p_0}\right)-f\bigl(\bar\theta(0)\bigr)\right]
 -\alpha d\leq 0,
\]
due to the fact that a convex function with finite recession constant $\alpha$ satisfies
\[ f(x+h)-f(x)\leq \alpha h,\qquad x,h\geq0.
\]

Moreover, since $S(0)\leq S(\rho)$,
\[m^{\widehat\nu} =m^{\bar\nu}-dS(\rho)+dS(0)\leq0.
\]
Thus, in either case, we obtain a probability measure $\nu$ on $[0,\rho)$ such that $m^\nu\leq0$ and $D_f(\nu \, \Vert \, \mu) \leq \liminf_{n\to\infty}D_f(\nu_n \, \Vert \, \mu)$. 

If $m^\nu=0$, then $\nu$ is already a minimizer. Otherwise, $m^\nu<0$, and since $m^\mu>0$, define 
\[ \nu^*:=(1-t)\nu+t\mu, \qquad t:=\frac{-m^\nu}{m^\mu-m^\nu}\in(0,1). 
\] Then $m^{\nu^*}=0$, and hence $\nu^*\in\mathcal P_0$. By convexity, 
\[ D_f(\nu^* \, \Vert \, \mu) \leq (1-t)D_f(\nu \, \Vert \, \mu)+tD_f(\mu \, \Vert \, \mu) \leq D_f(\nu \, \Vert \, \mu). 
\] 
Therefore, $\nu^*$ is a minimizer of problem~\eqref{eq:G}, and hence also of problem \eqref{eq:B} by Proposition \ref{prop:equiv}\textup{(i)}.

\bigskip 
We next prove \textup{(ii)}. Since $f'(\infty)=\infty$, every measure
with finite $f$-divergence is absolutely continuous with respect to $\mu$. 

By assumption, the $c$-sublevel set $\left\{\nu\in\mathcal P_0:\nu\succeq\mu ,\ D_f(\nu\Vert\mu)\leq c\right\}$
is nonempty, and hence $\inf_{\nu\in\mathcal P_0,\nu\succeq \mu}D_f(\nu\Vert\mu)\leq c$. If this holds with equality, any measure in this
sublevel set is already a minimizer. Thus, it remains only to consider
the case $\inf_{\nu\in\mathcal P_0,\nu\succeq \mu}D_f(\nu\Vert\mu)< c$. In this case, we may choose a minimizing sequence
$(\nu_n)$ for problem~\eqref{eq:C'} such that, after discarding finitely
many terms, $D_f(\nu_n\Vert\mu)\leq c$ for all $n$.
Since $f'(\infty)=+\infty$, each $\nu_n$ is absolutely continuous with
respect to $\mu$; write
\[
    \nu_n(\ud u)=\theta_n(u)\mu(\ud u).
\]
Therefore,
\[ \int_{[0,\rho)}\theta_n\,\ud\mu=1,
    \qquad
    \int_{[0,\rho)}S(u)\theta_n(u)\,\mu(\ud u)=0,
    \qquad
    \int_{[0,\rho)}f(\theta_n)\,\ud\mu= D_f(\nu_n\Vert\mu) \leq c.
\]
The de la Vall\'ee--Poussin criterion implies
that $(\theta_n)$ is uniformly integrable. By the Dunford--Pettis theorem,
after passing to a subsequence, there exists $\theta\in L^1_+$ such
that $\theta_n\to\theta$ weakly in $L^1$. In particular,
\[
    \int_{[0,\rho)}\theta\ud\mu=\lim_{n\to \infty}\int_{[0,\rho)}\theta_n\ud \mu=1.
\]
Define $\nu(\ud u):=\theta(u)\mu(\ud u)$. By lower semicontinuity, $D_f(\nu\Vert\mu)\leq\liminf_{n\to\infty}D_f(\nu_n\Vert\mu)$.

We next verify the scaled-mean constraint. For $K>0$, let
\[S_K:=S\wedge K.
\]
Since $S_K\in L^\infty(\mu)$,
\[\int S_K\theta\,\ud\mu=\lim_{n\to\infty}\int S_K\theta_n\,\ud\mu.
\]
Using $S=S_K+(S-K)_+$ and $m^{\nu_n}=0$, we obtain
\[\int S_K\theta_n\,\ud\mu=-\int(S-K)_+\theta_n\,\ud\mu.
\]
Consequently,
\[\left|\int S_K\theta\,\ud\mu\right| \leq\limsup_{n\to\infty}\int_{\{S>K\}}S(u)\theta_n(u)\,\mu(\ud u).
\]
By \eqref{eq:int}, the right-hand side tends to zero as
$K\to\infty$. Since $S_K-S(0)\uparrow S-S(0)$, the monotone convergence theorem yields
\[\int_{[0,\rho)}S(u)\theta(u)\,\mu(\ud u)=0.
\]

It remains to verify the first-order stochastic dominance constraint.
Since $\nu_n\succeq\mu$, for every $x\in[0,\rho)$,
\[\int_{[0,x]}\theta_n(u)\,\mu(\ud u)\leq\mu([0,x]).
\]
Since $\ind_{[0,x]}\in L^\infty(\mu)$, weak $L^1$ convergence gives
\[\int_{[0,x]}\theta(u)\,\mu(\ud u)=\lim_{n\to\infty}\int_{[0,x]}\theta_n(u)\,\mu(\ud u)\leq\mu([0,x]).
\]
Hence $\nu\succeq\mu$.
Therefore, $\nu$ is a minimizer of problem~\eqref{eq:C'}, and hence also of problem \eqref{eq:C} by Proposition \ref{prop:equiv}\textup{(ii)}.

\bigskip Finally, we prove \textup{(iii)}. Let $(\nu_n)$ be a minimizing sequence
for problem~\eqref{eq:CM'}. Since all $\nu_n$ are supported on the compact
interval $[0,M]$, after passing to a subsequence, $\nu_n\Rightarrow\nu$ for some probability measure $\nu$ supported on $[0,M]$.
Since $S$ is bounded and continuous on $[0,M]$,
\[\int_{[0,M]}S(u)\,\nu(\ud u)=\lim_{n\to\infty}\int_{[0,M]}S(u)\,\nu_n(\ud u)=0.
\]
Moreover, since first-order stochastic dominance is closed under weak
convergence, $\nu\succeq\mu$. Finally, lower semicontinuity gives $D_f(\nu\Vert\mu)\leq\liminf_{n\to\infty}D_f(\nu_n\Vert\mu)$.
Therefore, $\nu$ is a minimizer of problem~\eqref{eq:CM'}, and hence also of problem \eqref{eq:CM} by Proposition \ref{prop:equiv}\textup{(iii)}. Note that this argument does not depend on whether $f'(\infty)$ is finite or infinite.

\bigskip It remains to prove uniqueness. If $f'(\infty)=\infty$, every minimizer with finite divergence is absolutely continuous with respect to $\mu$. Thus, if $\ud \nu_i=\theta_i \,\ud \mu$, $i=1,2$, are two distinct minimizers, their average is feasible and strict convexity of $f$ gives 
\[ D_f\Bigl(\frac{\nu_1+\nu_2}{2} \, \Big\Vert \, \mu\Bigr) < \frac12D_f(\nu_1 \, \Vert \, \mu) + \frac12D_f(\nu_2 \, \Vert \, \mu), 
\] contradicting minimality. 
Under strict convexity, if a \eqref{eq:B}-optimizer had $m^\nu<0$, mixing it with $\mu$ to reach zero mean would strictly reduce the divergence. Hence every \eqref{eq:B}-optimizer has zero mean and coincides with the unique \eqref{eq:G}-optimizer.

If $f'(\infty)<\infty$, Lemma \ref{lem:0} shows that a minimizer must have its singular mass at the relevant endpoint, namely $0$ in \textup{(i)} and $M$ in \textup{(iii)}. Hence it can be written, respectively, as 
\[ 
\ud \nu_i=a_i\, \ud\delta_0+\theta_i \, \ud \mu \qquad\text{or}\qquad \ud \nu_i=a_i \, \ud\delta_M+\theta_i\, \ud \mu, 
\] with the endpoint mass absorbed into the density when the corresponding endpoint has positive $\mu$-mass. If $\theta_1\neq\theta_2$ on a set of positive $\mu$-measure, strict convexity again gives a strict improvement for their average, contradicting minimality. Therefore, $\theta_1=\theta_2$ $\mu$-a.e., which also implies $a_1=a_2$.
Hence $\nu_1=\nu_2$. Thus the minimizer is unique. 
\end{proof}

\begin{table}
    \centering
    \begin{tabular}{|c|c|c||c|c|c|}
    \hline
    \multicolumn{3}{|c||}{Conditions satisfied}
    & \multicolumn{3}{c|}{Existence of a minimizer}\\
    \hline
    $f'(\infty)=+\infty$
    & $S(\rho)=+\infty$
    & $r_0\in(\inf\supp\mu,\sup\supp\mu)$
    & \eqref{eq:B} & \eqref{eq:C}& \eqref{eq:CM}\\
    \hline\hline
    Yes & Yes & Yes & Yes & Yes, if UI & Yes\\
    \hline
    Yes & Yes & No & No & No & No\\
    \hline
    Yes & No & Yes & Yes & Yes & Yes\\
    \hline
    Yes & No & No & No & No & No\\
    \hline
    No & Yes & Yes & Yes & No & Yes\\
    \hline
    No & Yes & No & Yes & No & Yes\\
    \hline
    No & No & Yes & Yes & No & Yes\\
    \hline
    No & No & No & Yes & No & Yes\\
    \hline
    \end{tabular}
\caption{Summary of the existence results for Problems~\eqref{eq:B},
\eqref{eq:C}, and \eqref{eq:CM}. For Problem~\eqref{eq:B} we assume
$m^\mu>0$, while for Problems~\eqref{eq:C} and \eqref{eq:CM} we assume
$m^\mu<0$. In the existence columns, ``No'' means that existence is not guaranteed by the stated results; it does not always imply nonexistence. UI denotes the $S$-uniform integrability condition~\eqref{eq:int}.}
    \label{tab:overview}
\end{table}

Table \ref{tab:overview} summarizes the main affirmative results for problems \eqref{eq:B}, \eqref{eq:C}, and \eqref{eq:CM}. We now aim to determine the minimizers analytically using the calculus of variations.
To do so, we distinguish two main cases: the superlinear case
$f'(\infty)=\infty$, in which any minimizer with finite divergence is
necessarily absolutely continuous with respect to $\mu$, and the
non-superlinear case $f'(\infty)<\infty$, in which a singular part may occur. Before moving on, we first present a useful technical lemma that will be applied in the following subsections. The proof is relegated to Appendix A.
\begin{lemma}\label{lem:comonotonic}
Let $\mu$ be a probability measure on $[0,\rho)$, $S\in L^1$ be continuous and increasing, and $0\leq a\leq b\leq\infty$.
Suppose that $ \theta_0:\ [0,\rho)\to[0,\infty]$ satisfies
\[
    a\le \theta_0\le b,\qquad a\leq\int_{[0,\rho)} \theta_0\,\ud\mu=C\leq 1\quad\text{and}\quad S\theta_0\in L^1.
\]
Then there exists a non-decreasing function $h$ such that
\begin{align*}
& a\le h\le b, \qquad \int_{[0,\rho)} h\,\ud\mu=C,\quad\text{and}\\
& \int_{[0,\rho)} S(u)\theta_0(u)\,\mu(\ud u) \leq \int_{[0,\rho)} S(u)h(u)\,\mu(\ud u) <\infty.
\end{align*}
\end{lemma}

\subsection{Optimal measure for superlinear \texorpdfstring{$f$}{f}}

In this subsection, we discuss the solutions to both problems \eqref{eq:B'} and \eqref{eq:C'} with a superlinear generator $f$, i.e., $f'(\infty)=\infty$. In this case, any measure with a nonzero singular part has infinite divergence. Therefore, we assume that $\nu$ is absolutely continuous with respect to $\mu$, denoted by $\nu\ll\mu$ with $\theta=\frac{\ud \nu}{\ud \mu}\in L^1$, and then,
\[
	D_f(\nu \, \Vert \, \mu)=\int_{[0,\rho)} f\left(\frac{\ud\nu}{\ud\mu}\right)\ud\mu.
\]

To move from the general existence and uniqueness results given in
Theorem \ref{thm:eu} to analytical solutions for specific choices of
$f$ and $S$, we formulate the first-order condition in terms of its
ordinary subdifferential. Recall that we extend $f$ by setting $f(x)=\infty$ for $x<0$. Then, for $x\geq0$, we define its subdifferential by
\[\partial f(x):=\left\{z\in\mathbb R:f(y)\geq f(x)+z(y-x)\ \text{for all }y\in\mathbb R\right\},
\qquad \text{if } f(x)<\infty,
\]
and $\partial f(x):=\emptyset$ if $f(x)=\infty$. We denote by $L^1_+$ the cone of nonnegative functions in $L^1$.
\begin{prop}\label{prop:mainthm2}
	Assume $\theta\in L^1_+$, $S\theta\in L^1$  and constants $\lambda, \eta\in \mathbb R$ satisfy, for $\mu$-almost every $u$,
	\begin{align}
	\label{fdiv1}
& -\lambda -\eta S(u) \in \partial f\bigl(\theta(u)\bigr)\ \ ; \ \ \int_{[0,\rho)}\theta(u) \, \mu(\ud u)=1;\ \ \int_{[0,\rho)} S(u)\theta(u) \, \mu(\ud u)= 0,
	\end{align}
and define the measure $\nu(\ud x)=\theta(x)\mu(\ud x)$. Then $\nu$ solves Problem~\eqref{eq:G}. Moreover,
\begin{itemize}
\item[\textup{(i)}] If $m^\mu>0$, the problem \eqref{eq:G} has at least one solution $\hat{\nu}$ such that $\mu\succeq \hat{\nu}$.
\item[\textup{(ii)}] If $m^\mu<0$, the problem \eqref{eq:G} has at least one solution $\hat{\nu}$ such that $\hat{\nu}\succeq\mu$.
\end{itemize}
\end{prop}

\begin{proof}
We aim to minimize the functional $F(\theta):=\int_{[0,\rho)} f\bigl(\theta(u)\bigr)\mu(\ud u)$ over all $\theta\in L^1_+$ satisfying
the constraints $\int_{[0,\rho)}\theta(u)\mu(\ud u)=1$ and $\int_{[0,\rho)}S(u)\theta(u)\mu(\ud u)=0$.
If indeed $\theta$ satisfies the constraints and
$\lambda,\eta\in\mathbb R$ such that
\[-\lambda-\eta S(u)\in\partial f\bigl(\theta(u)\bigr)\qquad\text{for $\mu$-almost every }u.
\]
Then, by the pointwise convexity,
\[f\bigl(\theta'(u)\bigr)-f\bigl(\theta(u)\bigr) \geq\bigl(-\lambda-\eta S(u)\bigr)\bigl(\theta'(u)-\theta(u)\bigr)
\]
for every $\theta'$ that satisfies the constraints. Thus, by integration,
\begin{align*}
    F(\theta')-F(\theta)&\ge \int_{[0,\rho)}\bigl(-\lambda-\eta S(u)\bigr)\bigl(\theta'(u)-\theta(u)\bigr)\mu(\ud u) \\
    &=-\lambda \int_{[0,\rho)}\bigl(\theta'(u)-\theta(u)\bigr)\mu(\ud u)-\eta \int_{[0,\rho)} S(u)\bigl(\theta'(u)-\theta(u)\bigr)\mu(\ud u)=0,
\end{align*}
where the last equality follows from the two constraints in \eqref{fdiv1}. Therefore, $\theta$ is a minimizer.

We next prove the first-order stochastic dominance statements. Note that the subdifferential of a convex function is monotone, i.e., if $y_i\in\partial f(x_i)$, $i=1,2$, then $(y_1-y_2)(x_1-x_2)\geq0$. Hence, for $u<v$, 
\[\Big(-\lambda-\eta S(v)-\left(-\lambda-\eta S(u)\right)\Big) \bigl(\theta(v)-\theta(u)\bigr) \ge 0,
\]
which implies $-\eta\bigl(S(v)-S(u)\bigr) \bigl(\theta(v)-\theta(u)\bigr) \geq0$. Since $S$ is strictly increasing, $\theta$ is non-increasing when $\eta>0$ and non-decreasing when $\eta<0$.

We first complete the proof of (ii) for $m^\mu<0$. If $\eta>0$, then $\theta$ is non-increasing as a function of $u$. Since $S$ is increasing, Chebyshev's inequality (see \cite[Theorem 43]{HLP52}) for oppositely
ordered functions gives
\[
    \int_{[0,\rho)} S(u)\theta(u)\,\mu(\ud u)
    \leq
    \left(\int_{[0,\rho)} S(u)\,\mu(\ud u)\right)
    \left(\int_{[0,\rho)} \theta(u)\,\mu(\ud u)\right)
    =
    m^\mu<0,
\]
which contradicts $\int_{[0,\rho)} S(u)\theta(u)\,\mu(\ud u)=0$. Thus $\eta \leq 0$.

In the case $\eta < 0$, it follows from the monotonicity relation above that $\theta$ is non-decreasing as a function of $u$. Therefore, 
\[
    \int_{[0,\rho)} \bigl(\theta(u)-1\bigr)\,\mu(\ud u)=0
\]
implies
\[
F^\nu(x)-F^\mu(x)=\int_{[0,x]}\bigl(\theta(u)-1\bigr)\,\mu(\ud u)\leq 0, \quad x\in[0,\rho),
\]
i.e., $\nu\succeq\mu$, which yields (ii).

In the case $\eta = 0$, we have only the condition $-\lambda\in\partial f\bigl(\theta(u)\bigr)$,
which does not necessarily fully characterize $\theta$. Define $I_\lambda:=\left\{x\geq0:-\lambda\in\partial f(x)\right\}$, which is the set of minimizers of the convex function $x\mapsto f(x)+\lambda x$. Since $f$ is superlinear, $f(x)+\lambda x\to\infty$ as $x\to\infty$. Therefore, $I_\lambda$ is bounded from above. Moreover, since it is the
minimizer set of a lower semicontinuous convex function, it is a closed
interval. Thus, we may write $I_\lambda=[a,b]$ with $0\leq a\leq b<\infty$.

Since the $\theta$ in the statement of the proposition satisfies
$\theta(u)\in I_\lambda$ for $\mu$-almost every $u$, we have
\[
a\le \theta\le b,\qquad \int_{[0,\rho)} \theta \, \ud\mu=1,\qquad\int_{[0,\rho)} S(u)\theta(u) \, \mu(\ud u)=0.
\]
Therefore, $a\le1\le b$ and by Lemma \ref{lem:comonotonic}, we can find a non-decreasing function $h$ such that
\[
a\leq h\leq b,\qquad \int_{[0,\rho)} h \, \ud \, \mu=1,\qquad\int_{[0,\rho)} S(u)h(u) \, \mu(\ud u)\ge0.
\]
Since $m^\mu<0$, we can construct a density $\widehat \theta$ by interpolating between $1$ and $h$, setting
\[
\widehat\theta:=(1-t)+t h \quad\text{with} \quad t=\frac{-m^\mu}{\int_{[0,\rho)} S(u)h(u)\,\mu(\ud u)-m^\mu}\in(0,1];
\]
it satisfies
\[
a\leq \widehat\theta\le b, \qquad \int_{[0,\rho)} \widehat\theta \, \ud\mu=1,\qquad\int_{[0,\rho)} S(u)\widehat\theta(u) \, \mu(\ud u)=0.
\]
We have $-\lambda\in\partial f(\widehat\theta(u))$ and, clearly, $\widehat\theta$ is non-decreasing. Thus, defining $\widehat\nu(\ud u):=\widehat\theta(u) \, \mu(\ud u)$, we have $\widehat\nu\succeq\mu$, showing that $\widehat\nu$ solves problem \eqref{eq:G}.

The case $m^\mu>0$ follows by an analogous argument.
\end{proof}

\begin{remark}
    Note that the condition in Proposition \ref{prop:mainthm2} is only sufficient, we do not claim that every optimizer does have this form. Though we illustrate in Section \ref{sec:ex} that this condition is typically satisfied in common examples.
\end{remark}

\begin{remark}
\label{rem:3.9}
A sufficient condition to guarantee the non-emptiness of the feasible set $\mathcal{P}_0$ is to require $r_0\in\bigl(\inf\supp\mu,\sup\supp\mu\bigr)$. To see this, e.g., consider the case $m^\mu<0$. Then there exist unique constants $a\in(0,1)$ and $b>0$ such that the measure defined by
\[
\widetilde\nu(A):=\mu(A)-a\mu\bigl(A\cap[0,r_0)\bigr)+b\mu\bigl(A\cap(r_0,\rho)\bigr),\qquad A\in\mathcal B\bigl([0,\rho)\bigr),
\]
is a probability measure satisfying $m^{\widetilde\nu}=0$, i.e., $\widetilde\nu\in\mathcal P_0$. The case $m^\mu>0$ follows analogously.

This support condition makes sure the target distribution is neither too aggressive nor too conservative. 
If $\supp\mu$ lies strictly on one side of $r_0$, then every measure
absolutely continuous with respect to $\mu$ has scaled mean of the same
strict sign, and hence no finite divergence feasible measure with zero
scaled mean exists in the superlinear case.
\end{remark}

\begin{remark}
It is worth noting that the $S$-uniform integrability condition \eqref{eq:int} is a sufficient condition for problem \eqref{eq:C} to admit a solution, not a necessary one. Example \ref{ex:counter1} shows that without any condition of this type, existence of a solution might fail as the infimum may not be attained.
\end{remark}
\begin{corollary}\label{f}
	When $f$ is differentiable, then the first part of the condition $\eqref{fdiv1}$ can be expressed much nicer as $\theta$ solving
	\begin{equation}\label{eq:diff-form}
		f'\bigl(\theta(u)\bigr)+\lambda+\eta S(u)=0\qquad\text{for } \mu \text{-almost every}\ u\in \supp \mu
	\end{equation}
    on $\theta >0$.
\end{corollary}

\subsection{Optimal measures for non-superlinear \texorpdfstring{$f$}{f}}
In this subsection, we consider optimal measures under the $f$-divergence when $f$ is non-superlinear, that is, $f'(\infty)<\infty$. For notational simplicity, we denote $\alpha := f'(\infty)$. In this case, candidate minimizers need not be absolutely continuous with respect to $\mu$, and singular components with respect to $\mu$ may appear. Financially speaking, this means that the optimizer may not fully preserve the client's original preference range described by $\mu$, and may instead assign mass to portfolio outcomes outside the support of the target distribution.

In this case, the extended $f$-divergence is given by
\[D_f(\nu \, \Vert \, \mu)=\int_{[0,\rho)}f\Bigl(\frac{\ud\nu^{ac}}{\ud\mu}\Bigr) \ud \mu+f'(\infty)\nu^\perp\bigl([0,\rho)\bigr).
\]

We first present an important property of the optimization problem in the non-superlinear case: if a minimizer has a singular component, then we can choose a minimizer whose singular mass is concentrated at a single point.
\begin{lemma}\label{lem:0}
The following endpoint-reduction properties hold.
\begin{itemize}
\item[\textup{(i)}]
For every feasible measure $\nu$ of problem~\eqref{eq:B'}, there exists
a feasible measure $\widehat\nu$ of problem~\eqref{eq:B'} such that $D_f(\widehat\nu\Vert\mu)\leq D_f(\nu\Vert\mu)$ and whose singular component is concentrated at the ruin point $0$, i.e., $\widehat\nu^\perp((0,\rho))=0$.
In particular, if problem~\eqref{eq:B'} admits a minimizer, then it admits
a minimizer of this form.

\item[\textup{(ii)}]
For every feasible measure $\nu$ of problem~\eqref{eq:CM'}, there exists
a feasible measure $\widehat\nu$ of problem~\eqref{eq:CM'} such that $D_f(\widehat\nu\Vert\mu)\leq D_f(\nu\Vert\mu)$ and whose singular component is concentrated at the upper endpoint $M$,
i.e., $\widehat\nu^\perp([0,M))=0$.
In particular, if problem~\eqref{eq:CM'} admits a minimizer, then it admits
a minimizer of this form.
\end{itemize}
In the case that $f$ is strictly convex, the minimizers are unique.
\end{lemma}

\begin{proof}
For \textup{(i)}, let $\nu$ be any feasible measure for problem~\eqref{eq:B'}. If
$\nu^\perp((0,\rho))=0$, there is nothing to prove. Otherwise, we construct
a new measure $\widetilde\nu$ by moving this singular mass to the ruin point
$0$. More precisely, define
\[\widetilde\nu:=\nu^{ac}+\nu^\perp(\{0\})\delta_0+\nu^\perp((0,\rho))\delta_0.
\]
If $\mu(\{0\})=0$, clearly $D_f(\widetilde\nu\Vert\mu)=D_f(\nu\Vert\mu)$.
If $\mu(\{0\})>0$, then the moved mass becomes absolutely continuous at
the atom $0$, in which case the divergence still does not increase.
Specifically, a convex function $f$ with finite recession constant
$\alpha$ satisfies
\[f(x+h)-f(x)\leq\alpha h,\qquad x,h\geq0.
\]
Therefore,
\begin{equation}\label{eq:sing-move}
\Biggl(f\biggl(\frac{\nu^\perp((0,\rho))+\nu(\{0\})}{\mu(\{0\})}\biggr)-f\biggl(\frac{\nu(\{0\})}{\mu(\{0\})}\biggr)
\Biggr)\mu(\{0\})\leq
f'(\infty)\nu^\perp((0,\rho)).
\end{equation}
Hence, in either case, $D_f(\widetilde\nu\Vert\mu)\leq D_f(\nu\Vert\mu)$.

On the other hand, since $S$ is strictly increasing and
$S(0)<S(x)$ for all $x>0$, $m^{\widetilde\nu}<m^\nu=0$.
Thus, interpolating between $\widetilde\nu$ and $\mu$ exactly as in the
proof of Proposition~\ref{prop:equiv}, we obtain a feasible measure
$\widehat\nu$ with zero scaled mean such that
\[
D_f(\widehat\nu\Vert\mu)
\leq
D_f(\widetilde\nu\Vert\mu)
\leq
D_f(\nu\Vert\mu),
\]
and whose singular part is concentrated at $0$. Uniqueness for strictly convex $f$ follows from strict inequality in \eqref{eq:sing-move} in this case.

For \textup{(ii)}, the same argument applies by moving all singular mass
in $[0,M)$ to $M$. This can only increase the scaled mean and preserves
first-order stochastic dominance. If the resulting scaled mean is
positive, interpolating with $\mu$ restores the zero scaled-mean condition;
the interpolation preserves first-order stochastic dominance and does not
increase the $f$-divergence.
\end{proof}

\begin{remark}
Although an optimal measure does not need to be absolutely continuous with respect to $\mu$ when $f'(\infty)<\infty$, Lemma \ref{lem:0} provides an explicit characterization of its singular part. 

It is also worth noting that, whenever $M>r_0$, a
minimizer cannot be completely singular with respect to $\mu$, in other words, $\nu^{ac}\bigl([0,\rho)\bigr)>0$ for any minimizer $\nu$.
To see this, first observe that if $\nu\perp\mu$ is a probability measure, then $D_f(\nu \, \Vert \, \mu)=f(0)+f'(\infty)$.
Since $m^\mu<0$ and $M>r_0$, define
\[
    \bar\nu:=(1-p)\mu+p\delta_M,
    \qquad
    p:=\frac{-m^\mu}{S(M)-m^\mu}\in(0,1).
\]
Then $m^{\bar\nu}=0$. By the convexity of $f$ and the fact $ f(x+h)-f(x)\leq f'(\infty)h$ for all $x,h\geq0$, it is easy to obtain 
\[D_f(\bar\nu \, \Vert \, \mu)\leq p\bigl(f(0)+f'(\infty)\bigr)<D_f(\nu \, \Vert \, \mu)
\]
since $p\in(0,1)$ and $f(0)+f'(\infty)>0$. Under Assumption \ref{ass}, $f\geq0$ and $f(1)=0$, the equality
$f(0)+f'(\infty)=0$ would imply $f(0)=f'(\infty)=0$, and convexity would then force $f\equiv0$, which is a trivial case.

If $M=r_0$, then the only probability measure supported on $[0,M]$ with zero
scaled mean is $\delta_{r_0}$, which gives a trivial boundary case.
\end{remark}

We now apply a calculus of variations approach similar to that used in the
previous section. We first consider the case $m^\mu>0$. By Lemma~\ref{lem:0}, if a minimizer
exists, we may choose one of the form
\[
    \nu(\ud x)=\nu_0\delta_0(\ud x)+\theta(x)\mu(\ud x),
\]
where $\theta\in L^1_+$ and $\nu_0\geq0$. For simplicity, we assume
that $\mu(\{0\})=0$; if $\mu(\{0\})>0$, the mass at $0$ can be absorbed
into the absolutely continuous part; see Remark~\ref{rem:muM}.

Then, problem~\eqref{eq:G} can be written as
\[\inf_{\theta\in L^1_+}\left\{\int_{[0,\rho)}\bigl(f(\theta(u))-\alpha\theta(u)\bigr)\mu(\ud u)+\alpha\right\},
\]
subject to
\begin{equation}\label{eq:OM3}
\int_{[0,\rho)}\theta(u)\,\mu(\ud u)\leq1,\qquad\int_{[0,\rho)}\theta(u)\bigl(S(u)-S(0)\bigr)\mu(\ud u)=-S(0).
\end{equation}

Similarly, if $m^\mu<0$, we assume for simplicity that
$\mu(\{M\})=0$. If $\mu(\{M\})>0$, the mass at $M$ can be absorbed into
the absolutely continuous part. For $\supp\mu,\supp\nu\subseteq[0,M]$, $M<\rho$, problem \eqref{eq:CM'} without the stochastic dominance restriction can be written as
\[
    \inf_{\theta\in L^1_+}
    \left\{
        \int_{[0,M]}
        \bigl(f(\theta(u))-\alpha\theta(u)\bigr)\,\mu(\ud u)
        +\alpha
    \right\},
\]
subject to
\[
    \int_{[0,M]}\theta(u)\,\mu(\ud u)\leq1,
    \qquad
    \int_{[0,M]}
    \theta(u)\bigl(S(u)-S(M)\bigr)\,\mu(\ud u)
    =
    -S(M).
\]

We then have a result similar to Proposition~\ref{prop:mainthm2}.

\begin{prop}\label{prop:mainthm3} 
The following sufficient optimality conditions hold.
\begin{itemize}
\item[\textup{(i)}]
Suppose that $m^\mu>0$ and $\mu(\{0\})=0$. Assume that $\theta\in L^1_+$, $S\theta\in L^1$, $\lambda\geq0$, and $\eta\in\mathbb R$ satisfy, for $\mu$-almost every $u$,
\begin{equation}\label{eq:fdiv3}
\alpha-\lambda-\eta\bigl(S(u)-S(0)\bigr) \in\partial f\bigl(\theta(u)\bigr),
\end{equation}
together with
\[
\int_{[0,\rho)}\theta(u)\,\mu(\ud u)\leq1, \qquad\int_{[0,\rho)}\theta(u)\bigl(S(u)-S(0)\bigr)\,\mu(\ud u)=-S(0),
\]
and $\lambda\left(\int_{[0,\rho)}\theta(u)\,\mu(\ud u)-1\right)=0$.
Define
\[\nu_0:=1-\int_{[0,\rho)}\theta(u)\,\mu(\ud u) \quad\text{and}\quad \ud\nu:=\nu_0\ud \delta_0+\theta\ud\mu.
\]
Then, problem \eqref{eq:B'} has at least one solution $\widehat\nu$ such that $\mu\succeq\widehat\nu$.

\item[\textup{(ii)}]
Suppose that $m^\mu<0$. Fix $M\in[r_0,\rho)$ such that $\supp\mu\subseteq[0,M]$ and $\mu(\{M\})=0$. Assume that $\theta\in L^1_+$, $S\theta\in L^1$, $\lambda\ge0$, and $\eta\in\mathbb R$ satisfy, for $\mu$-almost every $u$,
\begin{equation}\label{eq:fdiv4}
\alpha-\lambda-\eta\bigl(S(u)-S(M)\bigr)
\in\partial f\bigl(\theta(u)\bigr),
\end{equation}
together with
\[ \int_{[0,M]}\theta(u)\,\mu(\ud u)\leq1, \qquad\int_{[0,M]}\theta(u)\bigl(S(u)-S(M)\bigr)\,\mu(\ud u)=-S(M),
\]
and $\lambda\left(\int_{[0,M]}\theta(u)\,\mu(\ud u)-1\right)=0$. Define
\[\nu_M :=1-\int_{[0,M]}\theta(u)\,\mu(\ud u)\quad\text{and}\quad \ud\nu:=\nu_M\ud\delta_M+\theta\ud\mu.
\]
Then, problem \eqref{eq:CM'} has at least one solution $\widehat\nu$ such that $\widehat\nu\succeq\mu$.
\end{itemize}
\end{prop}

\begin{proof}
We first prove the optimality statement in \textup{(i)}. We aim to minimize $J(\theta):= \int_{[0,\rho)}\bigl(f(\theta(u))-\alpha\theta(u)\bigr)\mu(\ud u)$. Let $\theta'$
be any other feasible density and set $A(u):=S(u)-S(0)$. By \eqref{eq:fdiv3} and the subgradient inequality,
\[
f\bigl(\theta'(u)\bigr)-\alpha\theta'(u)-\bigl(f(\theta(u))-\alpha\theta(u)\bigr)\geq-\lambda\bigl(\theta'(u)-\theta(u)\bigr) -\eta A(u)\bigl(\theta'(u)-\theta(u)\bigr).
\]
Integrating and using the equality constraint gives
\[
\begin{aligned}
 J(\theta')-J(\theta)\geq-\lambda\left(\int_{[0,\rho)}\theta'\ud\mu-\int_{[0,\rho)}\theta\ud\mu\right).
\end{aligned}
\]
If $\lambda=0$, the right-hand side is zero. If $\lambda>0$, the
complementarity condition implies $\int\theta\ud\mu=1$,
while feasibility gives $\int\theta'\ud\mu\leq1$. Hence in either case $J(\theta')\geq J(\theta)$.
Thus $\nu$ is a minimizer. A similar approach works for (ii).

To prove the first-order stochastic dominance, we consider case \textup{(ii)}; case \textup{(i)} follows analogously. From \eqref{eq:fdiv4}, there exists $\gamma(u)\in\partial f\bigl(\theta(u)\bigr)$ such that 
\[
\gamma(u) = \alpha-\lambda-\eta \bigl(S(u)-S(M)\bigr).
\]

We first claim that $\eta\leq0$. To prove this, suppose, toward a contradiction, that $\eta>0$. Since $S(u)-S(M)$ is strictly increasing, $\gamma(u)$ is strictly decreasing in $u$. Hence, by the monotonicity of the subdifferential of a convex function, $\theta$ is non-increasing, and for $v\geq u$ we have
\begin{equation}\label{eq:mono2}
-\eta\bigl(S(v)-S(u)\bigr)\bigl(\theta(v)-\theta(u)\bigr)\geq0.
\end{equation}

Define $c_\nu^\ll:=\int_{[0,M]}\theta(u)\,\mu(\ud u)\leq1$. If $c_\nu^\ll=1$, then $\nu$ has no singular mass at $M$. Since $S$ is increasing and $\theta$ is non-increasing, Chebyshev's inequality for oppositely ordered functions gives
\[0=m^\nu =\int_{[0,M]}S(u)\theta(u)\,\mu(\ud u)\leq\left(\int_{[0,M]}S(u)\,\mu(\ud u)\right)\left(\int_{[0,M]}\theta(u)\,\mu(\ud u)\right)=
 m^\mu<0,
\]
which is a contradiction.

If $c_\nu^\ll<1$, then the complementary slackness condition implies that $\lambda=0$. Hence
\[\gamma(u)=\alpha-\eta\bigl(S(u)-S(M)\bigr).
\]
For every $u<M$, since $S(u)<S(M)$ and $\eta>0$, we have $\gamma(u)>\alpha$.
On the other hand, since $\alpha=f'(\infty)$ is the recession slope of the convex function $f$, every finite subgradient of $f$ is bounded from above by $\alpha$. Thus, $\gamma(u)\in\partial f\bigl(\theta(u)\bigr)$ implies $\gamma(u)\leq\alpha$, which is again a contradiction.
Therefore, $\eta\leq0$.

If $\eta <0$, the monotonicity condition \eqref{eq:mono2} implies that $\theta$ is non-decreasing.

Assume now $\eta = 0$ and let
\[
g(x):=f(x)-\alpha x.
\]
Since $\eta=0$, $-\lambda\in\partial g(\theta(u))$. Denote $\theta\in[a,b]$, with $0\leq a\leq b\leq\infty$. Suppose that
$\theta_0$ is a feasible solution with $\eta=0$, namely,
\begin{equation}
\label{Condition}
a\le\theta_0\le b,\qquad \int_{[0,M]}\theta_0 \, \ud\mu\le1,\qquad \text{and} \qquad
    \int_{[0,M]} \bigl(S(M)-S(u)\bigr)\theta_0(u) \, \mu(\ud u)=S(M).
\end{equation}

By Lemma \ref{lem:comonotonic}, we can find a non-decreasing function $h$ that satisfies the above constraints and 
\[
\int_{[0,M]} \bigl(S(M)-S(u)\bigr)h(u) \, \mu(\ud u) \leq \int_{[0,M]} \bigl(S(M)-S(u)\bigr) \theta_0(u) \, \mu(\ud u)=S(M).
\]
Note that, due to the fact $ \beta:=\int_{[0,\rho)} \theta_0 \, \ud\mu\le1$ with inequality constraint instead of equality, we might not have $b\ge 1$. Therefore, take $\kappa:=\min\{b,1\}$. It is easy to verify that 
\[
K:=\kappa\int_{[0,M]} \bigl( S(M)-S(u) \bigr) \, \mu(\ud u) \ge S(M).
\]
We now construct
\[
\widehat\theta:=(1-t)h+t\kappa,\qquad t:=\frac{S(M)-\int_{[0,M]}(S(M)-S(u))h(u)\mu(\ud u)}{K-\int_{[0,M]}(S(M)-S(u))h(u)\mu(\ud u)}\in[0,1].
\]
By convention, simply take $t=0$ if the denominator is $0$. Then, it is straightforward to verify that $\widehat\theta$ satisfies all conditions in \eqref{Condition}.
As $a\le\widehat\theta\le b$ and $\int_{[0,\rho)} \widehat\theta \, \ud\mu=(1-t)\beta+t\kappa\le1$, $-\lambda\in\partial g(\widehat\theta(u))$, and $\widehat\theta$ is non-decreasing. Therefore
\[
\widehat\nu:=\widehat\theta\mu+\biggl(1-\int_{[0,M]} \widehat\theta\ud\mu\biggr)\delta_M,
\]
which implies $\widehat\nu\succeq\mu$.
Moreover, $\widehat\nu$ satisfies the optimality conditions above and
therefore solves problem \eqref{eq:CM'}.
\end{proof}

\begin{remark}
In contrast to problem~\eqref{eq:CM}, in the non-superlinear case problem~\eqref{eq:C} may also fail to attain its infimum. When $m^\mu<0$, the constraint $\int_{[0,\rho)} S\,\ud\nu=0 $ requires us to increase the scaled mean. In the case of $S(\rho) = \infty$, a minimizing sequence may be achieved in the following way: Choose an increasing sequence $(n_k)$, $\lim_{k \to \infty} n_k = \rho$ that satisfies $\mu\bigl(\{n_k\}\bigr) =0$ (as $\mu$ can only have countably many atoms, this can clearly be done). Define $\tilde{\nu}_k$ by setting $\tilde{\nu}_k(A) = \P\left(\max\{X, k\} \in A\right)$ for $X \sim \mu$ and $A \in \mathcal{B}\bigl([0,\infty)\bigr)$ and set 
\[
\epsilon_k := \frac{m^\mu}{m^\mu-\int_{[0,\rho)}S(u)\tilde{\nu}_{n_k}(\ud u)}, \qquad \ud \nu_k := (1-\epsilon_k) \ud \mu + \epsilon_k \ud\tilde{\nu}_{n_k}.
\]
Then $\nu_k$ is a sequence of probability measures with scaled mean zero that converge weakly to $\mu$ and satisfies $\nu_k \succeq \mu$. But as $m^\mu <0$, no minimizer can exist for separating $D_f$.

A similar argument works in the case $S(\rho) < \infty$, except that $\lim_{k \to \infty} \epsilon_k =: \epsilon > 0$. Thus the sequence $\nu_k$ converges again, albeit only vaguely. The limit is the sub-probability measure $(1-\epsilon)\mu$ which not only fails the scaled mean condition, but is even not a probability measure.
See also Examples~\ref{ex:counter1} and~\ref{ex:TV}.
\end{remark}

\begin{remark}\label{rem:muM}
In the cases $\mu\bigl(\{0\}\bigr)>0$ or $\mu\bigl(\{M\}\bigr)>0$, the corresponding endpoint
mass can be absorbed into the absolutely continuous part. Hence, Proposition~\ref{prop:mainthm3} reduces to the same type of problem as in Proposition \ref{prop:mainthm2}.
Indeed, Lemma \ref{lem:0} and Proposition \ref{prop:equiv} show that minimizers for
problem \eqref{eq:B} and the compactly supported
problem \eqref{eq:CM}, respectively, can be chosen to be
absolutely continuous with respect to $\mu$. Therefore, in both cases,
the extended $f$-divergence reduces to the ordinary $f$-divergence:
\[D_f(\nu \, \Vert \, \mu)= \int_{[0,\rho)} f\left(\frac{\ud\nu}{\ud\mu}\right)\,\ud\mu.
\]
\end{remark}

\begin{corollary}\label{coro2}
When $f$ is differentiable, then the first part of the condition~\eqref{eq:fdiv3} can be
expressed as $\theta$ solving 
\[
f'\bigl(\theta(u)\bigr) = \alpha-\lambda-\eta\bigl(S(u)-S(0)\bigr) \qquad\text{for } \mu \text{-almost all}\quad u\in\supp\mu,
\]
on $\theta >0$ and the first part of the condition~\eqref{eq:fdiv4} can be expressed as $\theta$ solving
\[
f'\bigl(\theta(u)\bigr) = \alpha-\lambda-\eta\bigl(S(u)-S(M)\bigr) \qquad\text{for } \mu \text{-almost all}\quad u\in\supp\mu,
\]
on $\theta >0$.
\end{corollary}

\section{Examples}\label{sec:ex}

In this section, we present explicit examples for the problems \eqref{eq:B}, \eqref{eq:C} and \eqref{eq:CM}, both analytically and numerically. We consider both superlinear ($f'(\infty)=\infty$) and non-superlinear ($f'(\infty)\in[0,\infty)$) choices of $f$. The examples illustrate the calculus of variations results established in Propositions~\ref{prop:mainthm2} and \ref{prop:mainthm3}, and provide counterexamples showing that, without appropriate assumptions, the infimum may fail to be attained, or that the first-order stochastic dominance condition in problems \eqref{eq:C} and \eqref{eq:CM} may not be satisfied.

For computational simplicity, throughout this section, the risk process is taken to be $R_t=B_t+2$, starting at $r_0=2$ so that the scale function is $S(x)=x-2$. We present here only the key results; lengthy calculations and verifications are omitted. 

Code for a numerical implementation based on distribution inputs via a graphical user interface can be found under \url{https://github.com/stesturm/Sale-Distribution-Builder}.

\paragraph{Kullback-Leibler divergence}\label{sec:KL}
The classical Kullback--Leibler divergence can be regarded as a special $f$-divergence with
\[
    f(x)=x\ln(x),    \qquad x\ge 0.
\]

\begin{example}\label{ex:counter1}
We give an explicit example showing that, when the $S$-uniform integrability condition \eqref{eq:int} in Theorem \ref{thm:eu} \textup{(ii)} fails, the infimum in problem \eqref{eq:C} is not attained. Consider
\[
\mu=\frac34\delta_0+\frac14\pi,
\] 
where $\pi$ is the Pareto distribution on $[1,\infty)$ with density $2x^{-3}\ind_{[1,\infty)}(x)$. Since $m^{\mu}<0$, this is a case of  problem \eqref{eq:C}.

For $n>1$ we construct the family of probability measures
 \[
 \nu_n := (1-\varepsilon_n)\mu+\varepsilon_n\mu_n \quad \text{with}\quad \varepsilon_n=\frac{3}{4n-1}\quad \text{and}\quad\mu_n(\ud x) := 4n^2\ind_{[n,\infty)}(x)\mu(\ud x) . 
 \]
Then, $m^{\nu_n}=0$, $\nu_n\ll\mu$ and $\mu_n\succeq\mu$ hence $\nu_n\succeq\mu$.
Thus $\nu_n$ is feasible for problem \eqref{eq:C}. 

However, by convexity, we obtain
\[
D_{\mathrm{KL}}(\nu_n \, \Vert \, \mu) \le \frac{3\log(4n^2)}{4n-1} \xrightarrow[]{n\to\infty}0.
\]
Thus the infimum in problem \eqref{eq:C} is zero. Since
$D_{\mathrm{KL}}(\nu \, \Vert \, \mu)=0$ if and only if $\nu=\mu$, the only
measure that could attain this value is $\mu$. But $\mu$ is not feasible for problem \eqref{eq:C} due to the fact that $m^\mu<0$. Therefore, problem \eqref{eq:C} does not admit a minimizer.
\end{example}

\paragraph{R\'enyi divergence; Hellinger and Bhattacharyya distances}\label{sec:Renyi}
For
$\alpha\in(0,1)\cup(1,\infty)$, the R\'enyi divergence is formally defined as
\[
D_R(\nu \, \Vert \, \mu):=\frac{1}{\alpha-1}\ln\int_{[0,\rho)}\left(\frac{\ud\nu}{\ud\mu}\right)^\alpha \, \ud \mu,
\]
whenever $\nu\ll\mu$. In the case $\alpha>1$ minimizing $D_R(\nu \, \Vert \, \mu)$ is equivalent to minimizing
\[
    \int_{[0,\rho)} \biggl(\frac{\ud\nu}{\ud\mu}\biggr)^\alpha \ud \mu.
\]
Thus, minimizing $D_R(\nu \, \Vert \, \mu)$ is equivalent to minimizing a
superlinear $f$-divergence with generator $f(x)=x^\alpha$.

\begin{example}\label{ex:renyi}
We take $\alpha=2$. Consider two target measures $\mu_1\sim\operatorname{Exp}\left(\frac13\right)$, and $\mu_2\sim\operatorname{Exp}(3)$. We can verify that $m^{\mu_1}>0$ and $m^{\mu_2}<0$, corresponding to problems \eqref{eq:B} and \eqref{eq:C}, respectively. The density of the optimal attainable distributions $\nu_1$ and $\nu_2$ is
\begin{align*}
f_{\nu_1}(x) &= \tfrac{1}{6}(2.7422 - 0.2544x)_+\, e^{-x/3}, \quad x>0,\\
f_{\nu_2}(x) & = \tfrac{3}{2} \left(327.5889(x - 2) + 218.3926\right)_+ e^{-3x}, \quad x>0,
\end{align*}
see Figure~\ref{fig:f1}.

\begin{figure}[htbp]
    \centering
    \includegraphics[width=0.48\textwidth, valign=b]{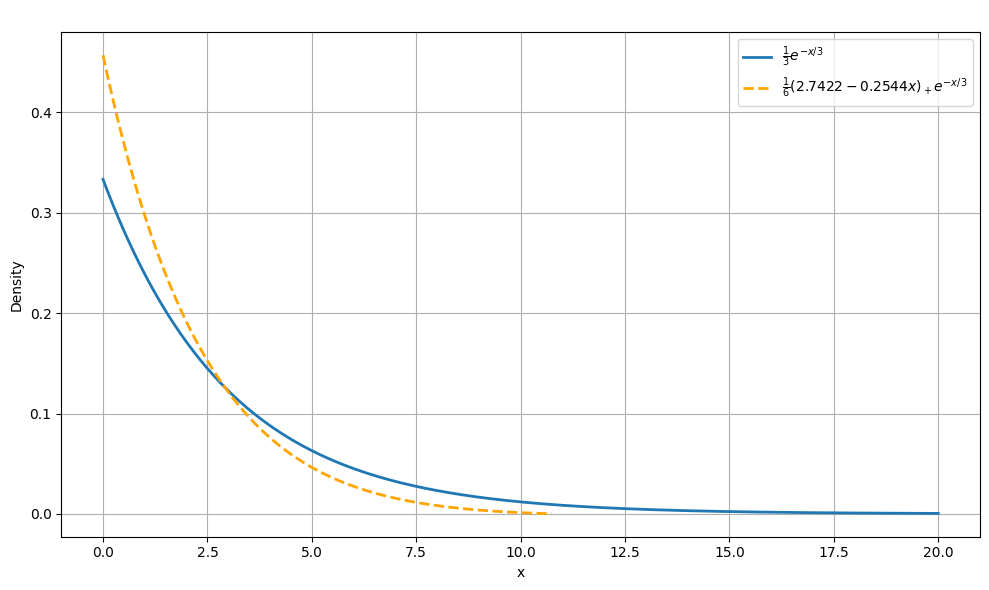} \quad \includegraphics[width=0.47\textwidth, valign=b]{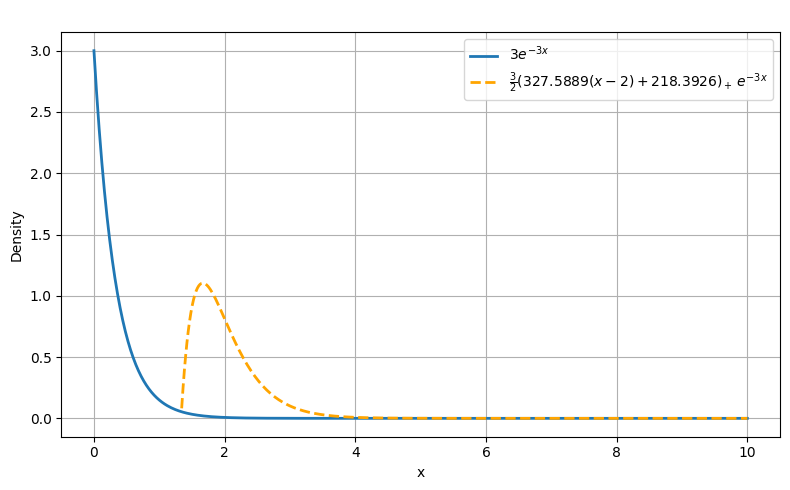}
    \caption{Example~\ref{ex:renyi}, R\'{e}nyi divergence: finding the density of attainable payoff closest to the infeasible target with  $m^{\mu_1} > 0$ (left) and increasing the payoff of a suboptimal distribution with $m^{\mu_2} < 0$ to the closest optimal one (right).}\label{fig:f1}
\end{figure}
\end{example}

In the case $0<\alpha<1$, since $1/(\alpha-1)<0$, minimizing $D_R(\nu \, \Vert \, \mu)$ is equivalent to minimizing 
\[
1-\int_{[0,\rho)}\theta(u)^\alpha\,\mu(\ud u).
\]
The R\'enyi divergence is closely related to the (squared) Hellinger distance,
\[
H^2(\mu,\nu) :=\frac12\int_{[0,\rho)} \Bigl(\sqrt{p(x)}-\sqrt{q(x)}\Bigr)^2\,\xi(\ud x),
\]
and the Bhattacharyya distance,
\[
D_B(\nu,\mu) = -\ln\bigl(\operatorname{BC}(\mu,\nu)\bigr), \qquad  \operatorname{BC}(\mu,\nu):= \int_{[0,\rho)} \sqrt{p(u)q(u)}\,\xi(\ud u),
\]
where
\[
\mu(\ud x)=p(x) \, \xi(\ud x),\quad\nu(\ud x)=q(x) \, \xi(\ud x).
\]
Minimizing the squared Hellinger distance or the Bhattacharyya
distance is equivalent to minimizing the R\'enyi divergence with
$\alpha=\frac12$.

\paragraph{Total variation distance and distances of affine type}\label{sec:TV}
The \textit{total variation distance} is an $f$-divergence with $f(x)=\frac12|x-1|$, implying $f'(\infty)=\frac12$,
\[
D_f(\nu \, \Vert \, \mu)=\frac12\int_{[0,\rho)} |\theta(u)-1|\mu(\ud u)+\frac12\nu^\perp\bigl([0,\rho)\bigr)=
|\nu-\mu|_{\mathrm{TV}}.
\]

\begin{example}\label{ex:TV}
Take
\[
\mu_1=\frac13\delta_1+\frac13\delta_3+\frac13\delta_5,
\qquad
\mu_2=\frac13\delta_0+\frac13\delta_1+\frac13\delta_2.
\]
Then $m^{\mu_1}=1>0$, and $m^{\mu_2}=-1<0$ corresponding to problems \eqref{eq:B} and \eqref{eq:CM}, respectively. By Proposition~\ref{prop:mainthm3}, the minimizers for problem \eqref{eq:B} and problem \eqref{eq:CM} with $M = 5$ are
\[
\nu_1^*=\frac15\delta_0+\frac13\delta_1+\frac13\delta_3+\frac{2}{15}\delta_5,
\qquad
\nu_2^*=\frac{2}{15}\delta_0+\frac13\delta_1+\frac13\delta_2+\frac15\delta_5.
\]
However, problem \eqref{eq:C} does not admit a solution, since its infimum is not attained. Indeed, for $N>3$, define
\[
\nu_N=\left(\frac13-\frac1N\right)\delta_0+\frac13\delta_1+\frac13\delta_2+\frac1N\delta_N.
\]
Then $m^{\nu_N}=0$, $\nu_N\succeq\mu_2$, and
\[
|\nu_N-\mu_2|_{\mathrm{TV}}=\frac1N\xrightarrow[]{N\to\infty}0.
\]
Since $\mu_2$ is the only measure satisfying
$D_f(\cdot \, \Vert \, \mu_2)=0$ but is not feasible, the infimum is not attained.
\end{example}

\begin{remark}
Although the generator of the total variation distance is not strictly
convex, the minimizers in Example~\ref{ex:TV} are unique. Indeed, for
$\mu_1$, any feasible measure must decrease the scaled mean by at least
$1$. Since
\[
\max_{x\in\supp\mu_1}\bigl(S(x)-S(0)\bigr)
=S(5)-S(0)=5,
\]
moving a total mass $t$ can decrease the scaled mean by at most $5t$.
Hence $t\geq1/5$, and equality is possible only by moving mass from $5$
to $0$. This gives uniquely
\[
\nu_1^*
=\frac15\delta_0+\frac13\delta_1+\frac13\delta_3+\frac{2}{15}\delta_5.
\]
The same argument applies to $\mu_2$.
Thus, strict convexity of $f$ is sufficient but not necessary for
uniqueness of the optimizer.
\end{remark}

In general, divergences with not strictly convex $f$ may in general lead to non-uniqueness, and some minimizers of the relaxed problem may fail to satisfy the first-order stochastic dominance constraint, as the next example shows.

\begin{example}\label{ex:FSD}
Set $M=3$ and consider the two-point distribution 
\[
\mu := \frac14\delta_0+\frac34\delta_1\quad\text{with}\quad m^\mu = -\frac54<0,
\] 
and aim to solve problem \eqref{eq:CM} with piecewise-affine $f$-divergence
\[
f(x) := \Bigl(x-\frac1{20}\Bigr)^+-\frac{19}{20}, \qquad x\ge0.
\] 
Here $f$ is convex but not strictly convex. Two minimizers of the compact zero-scaled-mean problem obtained from
\eqref{eq:CM'} by dropping the stochastic dominance constraint are
\[
\nu_1 = \frac3{10}\delta_0 + \frac1{20}\delta_1 + \frac{13}{20}\delta_3, \quad\text{and}\quad \nu_2 = \frac14\delta_0 + \frac18\delta_1 + \frac58\delta_3.
\] 
Both solutions solve the compact optimization problem, however,
\[
F^{\nu_1}(0) = \frac3{10} > \frac14 = F^\mu(0),
\] 
whence $\nu_1\not\succeq\mu$ while $\nu_2\succeq\mu$.
This example shows that, when $f$ is not strictly convex, not every
minimizer of the relaxed compact problem need satisfy the first-order
stochastic dominance constraint. In particular, although $\nu_2$ solves
problem~\eqref{eq:CM}, $\nu_1$ does not.
\end{example}

\section{Conclusion}\label{sec:conc}

In the current article we have considered the problem of the optimal timing of an asset sale under the distribution builder methodology for an asset that can default, using general (non-singular) diffusions as models for asset prices. We derived explicit characterizations of distributions that can be attained or super-attained in this model, building upon the theory of the Skorokhod embedding problem. To the best of our knowledge, this is the first extension of the Skorokhod problem to stopped diffusion processes.

We then complemented the distribution builder framework by presenting a method of how to determine the closest feasible distribution if the originally chosen distribution is not (super-)attainable. Similarly, if the original distribution is (super-)attainable but not optimal, we show how to determine the optimal distribution closest to the originally specified one. This analysis is conducted in a framework of extended $f$-divergences for the notion of distributional discrepancy, a rich class that encompasses many often-used discrepancies such as the Kullback--Leibler and R\'{e}nyi divergences as well as the total variation distance. The theoretical analysis is complemented with several examples to show that optimal distribution can be calculated easily in practice, as well as counterexamples to show that the conditions in our framework are indeed required.

\appendix
\section{Proof of Lemma \ref{lem:comonotonic}}
\begin{proof}
If $a=b=1$, then the result is trivial as $\theta_0$ is already non-decreasing; hence we assume $a<b$. We first prove the case $b<\infty$. Define
\[
g(u):=\frac{\theta_0(u)-a}{b-a},
\]
which satisfies
\[
0\le g\le 1,\qquad \int_{[0,\rho)} g \, \ud\mu=\frac{C-a}{b-a}=:q\in[0,1].
\]

Choose $c\in[0,\rho)$ such that
\[
\mu\bigl((c,\rho)\bigr)\le q\le \mu\bigl([c,\rho)\bigr),
\]
and $\gamma\in[0,1]$ such that
\[
\mu\bigl((c,\rho)\bigr)+\gamma\mu\bigl(\{c\}\bigr)=q,
\]
and define $g^*(u):= \ind_{\{u>c\}}+\gamma\ind_{\{u=c\}}$. Then $g^*$ is non-decreasing and satisfies
\[
0\le g^*\le1, \qquad\int_{[0,\rho)} g^*\,\ud\mu=q, \quad\text{and}\quad \int_{[0,\rho)} S(u)g^*(u)\,\mu(\ud u)\ge\int_{[0,\rho)} S(u)g(u)\,\mu(\ud u)
\]
for all $g$ satisfying the condition as all $g$ share the same mean and therefore
\[
\int_{[0,\rho)} S(u)\bigl(g^*(u)- g(u) \bigr) \,\mu(\ud u) = \int_{[0,\rho)} \bigl(S(u)-S(c) \bigr)\bigl(g^*(u)- g(u) \bigr) \,\mu(\ud u) \geq 0
\]
as the integrand of the right integral is, as product of non-negative or non-positive numbers, always nonnegative. Setting
\[
h(u):=a+(b-a)g^*(u)
\]
defines a non-decreasing function satisfying
\[
a\le h\le b, \qquad\int_{[0,\rho)} h\,\ud\mu=C,\quad\text{and}\quad\int_{[0,\rho)} S(u)h(u)\,\mu(\ud u) \ge \int_{[0,\rho)} S(u)\theta_0(u)\,\mu(\ud u)
\]
as desired.

Now suppose that $b=\infty$, if $C=a$, then $\theta_0=a$ trivially, and we simply set $h\equiv a$. If $C>a$, we define
\[
h(u):=a+\frac{C-a}{\mu\bigl([c,\rho)\bigr)}\ind_{\{u\ge c\}},
\]
where $c\in[0,\rho)$ is chosen so that 
\[ 
\frac1{\mu\bigl([c,\rho)\bigr)}\int_{[c,\rho)} S(u)\,\mu(\ud u) \ge \frac{1}{C-a}\int_{[0,\rho)} S(u)\bigl(\theta_0(u)-a\bigr)\,\mu(\ud u).
\]
Indeed,  such a $c$ exists in $[0,\rho)$: Denote
\[
L:=\frac{1}{C-a}\int_{[0,\rho)}S(u)\bigl(\theta_0(u)-a\bigr)\,\mu(\ud u)=\int_{[0,\rho)}S(u)\,\widetilde\mu(\ud u),
\]
for probability measure 
\[
\widetilde\mu(\ud u):= \frac{\theta_0(u)-a}{C-a}\,\mu(\ud u).
\]
on $[0,\rho)$. Since $S$ is not constant, $\widetilde\mu\bigl(\{S\ge L\}\bigr)>0$, implying $\mu\bigl(\{S\ge L\}\bigr)>0$ as $\widetilde\mu\ll\mu$. So, as $S$ is increasing and continuous, we can define 
\[
    c:=\inf\{x\in[0,\rho):S(x)\ge L\}<\rho
\]
and have $[c,\rho)=\{u:S(u)\ge L\}$.

Clearly, $h$ is non-decreasing and satisfies
\[
h\ge a,\qquad \int_{[0,\rho)} h\,\ud\mu=C,
\]
as well as
\begin{align*}
\int_{[0,\rho)} S(u)h(u)\,\mu(\ud u)&=a\int_{[0,\rho)} S(u)\mu(\ud u)+\frac{C-a}{\mu\bigl([c,\rho)\bigr)}\int_{[c,\rho)} S(u)\,\mu(\ud u) \\ &\ge \int_{[0,\rho)} S(u)\theta_0(u)\,\mu(\ud u).
\end{align*}
Finally, as $h$ is bounded and $S$ integrable, the integral on the left is finite.
\end{proof}
\section*{Acknowledgments}
AI (GPT 5.6 Sol) was used for a final error and typo check. The authors are responsible for all remaining mistakes.

\bibliographystyle{siamplain}
\bibliography{biblist}

\end{document}